\documentclass[12pt,reqno,a4paper]{amsart}
\usepackage[margin=.5in]{geometry}
\usepackage{amsmath,amsthm,pdfsync,verbatim,graphicx,epstopdf,enumerate}
\usepackage{placeins}
\usepackage{amsfonts}
\usepackage{amssymb,amsthm,amsmath}
\usepackage{amsmath}
 \usepackage{enumerate}
\usepackage{ifthen}
\usepackage{graphicx}
\usepackage{mathtools}
\mathtoolsset{showonlyrefs}
\usepackage{cite}
\usepackage{tikz} 
\usetikzlibrary{arrows.meta}
\usepackage{comment}
\usepackage{amsmath,amscd,amssymb}
\usepackage{latexsym}
\usepackage[colorlinks,citecolor=blue,pagebackref,hypertexnames=false]{hyperref}

\numberwithin{equation}{section}
\allowdisplaybreaks[2]
\theoremstyle{plain}
\newtheorem{theorem}{Theorem}[section]

\newtheorem{lemma}[theorem]{Lemma}

\newtheorem{proposition}[theorem]{Proposition}

\theoremstyle{definition}
\newtheorem{definition}[theorem]{Definition}

\theoremstyle{remark}
\newtheorem{remark}[theorem]{Remark}

\newtheorem{case[theorem]}{Case}

\def\norm#1.#2.{\lVert#1\rVert_{#2}}

\title[Mixed Grushin Heat Equations ]{Mixed Grushin Heat Equations with Riesz-Potential Nonlinearities in Marcinkiewicz Spaces: Subcritical, Critical, and Supercritical Regimes }

\author{Aparajita Dasgupta}
\author{Uttam Kumar Dolai}

 \address{\endgraf Department of Mathematics
 Indian Institute of Technology, Delhi, Hauz Khas
New Delhi-110016
 India}

\email{adasgupta@maths.iitd.ac.in}
\address{\endgraf Department of Mathematics
Indian Institute of Technology, Delhi, Hauz Khas
New Delhi-110016
India}
 \email{uk749043@gmail.com}

\subjclass[2020]{}
 
\date{\today}

\begin{document}

\maketitle
\begin{abstract}
We study the nonlinear evolution equation
$$
\partial_t u+(G+G^\delta)u=I_\alpha(|u|^\rho)
\quad\text{on }\mathbb{R}^{N+k},
$$
where \(G\) is the nonnegative self adjoint realization of Grushin operator, \(0<\delta<1\), and \(I_\alpha\) is the potential operator with kernel \(|z|^{-\alpha}\), \(0<\alpha<N+k\). We develop a well-posedness theory for initial data in Marcinkiewicz spaces, allowing singular profiles outside the corresponding Lebesgue spaces. Using spectral calculus, subordination, and interpolation, we establish Lebesgue and Lorentz smoothing estimates for the mixed semigroup \(e^{-t(G+G^\delta)}\), which distinguish the second-order behaviour at short times from the fractional decay at large times. In the subcritical regime, considering
$$
\frac1\beta=\frac2Q+1-\frac{\alpha}{d},
\qquad
\frac1{\beta_\delta}=\frac{2\delta}{Q}+1-\frac{\alpha}{d},
$$
where $Q$ is homogenous dimension, we prove local existence, uniqueness, Lipschitz dependence on the initial data, and a blow-up alternative in \(L^{p,\infty}\) when
$$
\rho<1+\frac{p}{\beta}.
$$
At the critical relation \(\rho=1+p/\beta\), a Yamazaki-type estimate and Lorentz duality yield global mild solutions for sufficiently small data in \(L^{p,\infty}\). In the supercritical regime \(\rho>1+p/\beta\), well-posedness is recovered in higher-integrability Marcinkiewicz spaces \(L^{q,\infty}\) satisfying
$$
\beta(\rho-1)<q<\beta_\delta(\rho-1),
$$
with global existence for sufficiently small initial data. In all cases, the solutions attain their initial data in the weak-\(*\) sense. These results extend the Marcinkiewicz-space theory to mixed Grushin diffusion with a spatially nonlocal potential source and reveal the role of the two competing diffusion scales in determining the admissible integrability regimes.
\end{abstract}

	\allowdisplaybreaks

	\tableofcontents

\section{Introduction}

The competition between diffusion and nonlinear growth is a fundamental
theme in the analysis of parabolic equations. Diffusion spreads and
regularizes the initial profile, whereas a superlinear source can
amplify it and lead to finite-time blow-up. The balance between these
mechanisms depends on the smoothing properties of the linear evolution
and the integrability of the nonlinear term. For degenerate diffusion
and spatially nonlocal sources, it also depends on the geometry of the
operator and the decay of the interaction kernel. In this paper, we
study a heat equation driven by a mixed local and nonlocal Grushin
operator with a Riesz-potential source. We develop a local and global
solvability theory in Marcinkiewicz spaces, allowing singular initial
data outside the corresponding strong Lebesgue spaces.

Let $z=(x,y)\in\mathbb{R}^{N}\times\mathbb{R}^{k}$, where $N,k\geq1$,
and set
\[
    d=N+k,\qquad Q=N+2k.
\]
We consider the Grushin operator
\[
    \Delta_G=-\frac12\left(\Delta_x+|x|^2\Delta_y\right),
\]
where $\Delta_x$ and $\Delta_y$ are the classical Laplacians in the
indicated variables. Introduced by Grushin~\cite{MR279436}, this
operator provides a basic model of a degenerate elliptic operator
satisfying H\"ormander's bracket condition. Indeed, its second-order
part is generated by the vector fields
\[
    X_i=\partial_{x_i},\qquad
    Y_{ij}=x_i\partial_{y_j},
    \qquad 1\leq i\leq N,\quad 1\leq j\leq k.
\]
Although the fields $Y_{ij}$ vanish on $\{x=0\}$, the commutators
$[X_i,Y_{ij}]=\partial_{y_j}$ recover the missing directions. This
structure explains how hypoellipticity persists despite the loss of
ellipticity along the degeneracy set. The operator is therefore a
natural setting in which to examine the effect of anisotropic
diffusion on nonlinear evolution. We refer to
\cite{MR2927124,MR2778943,MR3459626,MR3326328} for related geometric,
spectral, and heat-kernel results.

The underlying anisotropy is described by the dilations
\[
    D_\lambda(x,y)=(\lambda x,\lambda^2y),\qquad \lambda>0.
\]
Under these dilations, the Grushin operator has homogeneous degree
two and Lebesgue measure scales by $\lambda^Q$. Consequently, its
global heat-semigroup estimates involve the homogeneous dimension
$Q$, which differs from the topological dimension $d$. Moreover,
the dependence of the coefficients on $x$ prevents the heat semigroup
from being represented as an ordinary Euclidean convolution in all
variables. These features distinguish the Grushin problem from a
translation-invariant diffusion equation and make semigroup estimates
adapted to the operator essential.

We denote by $G$ the nonnegative self-adjoint realization of
$\Delta_G$ on $L^2(\mathbb{R}^{d})$. Applying the partial Fourier
transform in the $y$-variables relates $G$ to the family of scaled
Hermite operators
\[
    \frac12H(\lambda)
    =\frac12\left(-\Delta_x+|\lambda|^2|x|^2\right),
    \qquad \lambda\in\mathbb{R}^{k}.
\]
This correspondence provides a spectral description of $G$ and
allows its fractional powers to be defined by the functional
calculus. In particular, the scalar transform constructed by
Stempak~\cite{MR5099478} combines the partial Fourier transform
with scaled Hermite expansions and realizes the Grushin operator
as a multiplication operator on an associated spectral space.
For $0<\delta<1$, we consider the mixed operator
\[
    \mathcal{G}=G+G^\delta.
\]

Mixed operators have received considerable attention because they
combine diffusion mechanisms acting at different scales. In the
Euclidean setting, the operator
\[
    \mathcal{L}_{a,b}=-a\Delta+b(-\Delta)^\delta,
    \qquad a,b>0,
\]
corresponds to the combination of Brownian diffusion and a jump
process. Related operators occur in dispersal models that allow
both local movement and long-range relocation; see
\cite{MR2915668,MR3332849,MR4249816,MR4651677}.
The Grushin analogue provides a way to combine local and nonlocal
diffusion within a common degenerate geometry: $G^\delta$ is defined
from $G$ itself, and its semigroup is obtained by subordination of
the Grushin heat semigroup. At the spectral level, the multiplier
$\mu+\mu^\delta$ is dominated by $\mu$ at high frequencies and by
$\mu^\delta$ at low frequencies. The semigroup estimates therefore
reflect the two diffusion orders through different bounds at small
and large times.

The nonlinear problem studied here is
\begin{equation}\label{Grushin1}
    \begin{cases}
        \partial_tu+\mathcal{G}u=I_\alpha(|u|^\rho),
            & z\in\mathbb{R}^{d},\quad t>0,\\
        u(0)=u_0,
    \end{cases}
\end{equation}
where $\rho>1$ and $0<\alpha<d$. We use the convention
\begin{equation}\label{Rieszpotential}
    I_\alpha f(z)
    =A_\alpha\int_{\mathbb{R}^{d}}
        \frac{f(\zeta)}{|z-\zeta|^\alpha}\,d\zeta,
    \qquad
    A_\alpha=
    \frac{\Gamma(\alpha/2)}
    {2^{d-\alpha}\pi^{d/2}\Gamma((d-\alpha)/2)},
\end{equation}
whenever the integral is defined. Thus $\alpha$ is the decay
exponent of the kernel; the corresponding Riesz potential has
order $d-\alpha$ in the usual order-based notation
\cite{MR350027,MR1918790}.

Equation~\eqref{Grushin1} contains two different sources of
nonlocality. The fractional diffusion term redistributes the
solution through the spectral geometry of $G$, while the nonlinear
potential aggregates the profile $|u|^\rho$ over Euclidean space.
The latter mechanism is spatially nonlocal even though the source
depends only on the solution at the current time. Its control
requires an integrability gain from the potential estimate as
well as smoothing by the diffusion. A central question is whether
these two effects can be combined to construct solutions for
initial data with borderline integrability, and whether such
solutions remain globally bounded when the initial norm is small.

The classical point of departure is Fujita's
work~\cite{MR214914} on the semilinear heat equation
\[
    u_t-\Delta u=u^\rho.
\]
For nontrivial nonnegative initial data, the exponent
$1+2/d$ separates the range in which global nonnegative solutions
cannot exist from the range admitting global solutions for suitable
small data. The critical case was subsequently resolved through
the work of Hayakawa~\cite{MR338569} and
Sugitani~\cite{MR470493}. Weissler~\cite{MR599472} further
developed the connection between global existence and the
integrability of the initial datum, obtaining small-data results
in the critical Lebesgue space with exponent
$p_c=d(\rho-1)/2>1$. This viewpoint places the initial-data space
at the center of the analysis: the nonlinear exponent alone does
not describe the range of admissible profiles.

For fractional diffusion $(-\Delta)^\delta$, the corresponding
Fujita exponent is $1+2\delta/d$; see
\cite{MR239379,MR470493,MR2920620,MR1732885,MR1854046}.
For Euclidean mixed diffusion, Biagi, Punzo, and
Vecchi~\cite{MR4849503} and Del Pezzo and
Ferreira~\cite{MR4860154} established that the Fujita threshold
for a local power source is likewise governed by the fractional
component. These results show that adding an order-two diffusion
term does not remove the influence of the nonlocal component on
global solvability. They also motivate a careful distinction
between the estimates needed near the initial time and the
behavior of the diffusion at large times.

Spatial convolution sources lead to further changes in this
balance. Filippucci and Ghergu~\cite{MR4385777,MR4400581}
studied parabolic inequalities with nonlinear convolution terms,
while Fino and Torebek~\cite{MR5102211} investigated fractional
parabolic equations with Hartree-type nonlinearities. The more
specific source $I_\alpha(|u|^\rho)$ was studied in
\cite{MR5037982}, where the Fujita critical exponent is shown
to differ from the value suggested by the usual scaling argument.
These results highlight the role of the spatial interaction kernel
in determining nonlinear solvability and the need to distinguish
critical integrability conditions from sharp Fujita thresholds.

For the Grushin operator, Oliveira and
Viana~\cite{MR5073613} obtained heat-semigroup estimates in
Lebesgue spaces and studied the associated semilinear Cauchy
problem. Kogoj, Lima, and Viana~\cite{MR5014751} subsequently
developed a Marcinkiewicz-space theory for the local power
nonlinearity $|u|^{\rho-1}u$. In particular, they obtained
global mild solutions for small initial data at
\[
    p=\frac Q2(\rho-1),
\]
together with results on positivity, symmetry, and self-similarity.
Their work shows that Lorentz-space methods can accommodate
singular profiles in the Grushin setting and provides an analytical
starting point for the problem considered here.

The choice of Marcinkiewicz spaces is motivated both by the
initial data and by the nonlinear estimate. For $1<p<\infty$,
the space $L^{p,\infty}(\mathbb{R}^{d})$, also called weak-$L^p$,
strictly contains $L^p(\mathbb{R}^{d})$. For example, the
anisotropically homogeneous profile
\[
    u_0(x,y)=\varepsilon
        \bigl(|x|^4+|y|^2\bigr)^{-Q/(4p)},\qquad \varepsilon>0,
\]
belongs to $L^{p,\infty}(\mathbb{R}^{d})$ but not to
$L^p(\mathbb{R}^{d})$. Its weak-$L^p$ norm can nevertheless
be made arbitrarily small by choosing $\varepsilon$ small.
Thus a weak-$L^p$ existence theorem enlarges the
admissible data class. In addition, the kernel
$|z|^{-\alpha}$ belongs to $L^{d/\alpha,\infty}$, so
Lorentz spaces provide a natural setting for estimating the
potential term. This compatibility makes it possible to treat
the singularity of the datum and the weak integrability of the
interaction kernel within the same functional framework.

{The novelty of the present work lies in developing a
Marcinkiewicz-space solvability theory for the simultaneous
presence of mixed Grushin diffusion and a spatial Riesz-potential
source. Compared with the pure Grushin equation with a local
power nonlinearity studied in \cite{MR5014751}, our problem
requires control of a spatial convolution source under the
evolution generated by $G+G^\delta$. Compared with the Euclidean
mixed-diffusion results \cite{MR4849503,MR4860154} and the
Riesz-source problem in \cite{MR5037982}, the diffusion here
is degenerate and anisotropic, and the initial data need only
belong to weak-$L^p$. The resulting theory treats a combination
of geometric, spectral, and nonlinear features that is not
covered by these earlier results.}

A distinctive feature of the analysis is the interaction between
the homogeneous dimension $Q=N+2k$, which governs Grushin
smoothing, and the Euclidean dimension $d=N+k$, which enters
the potential estimate. Their simultaneous appearance changes
through a simple substitution of $Q$ for the Euclidean
dimension. Moreover, the two diffusion orders produce different
semigroup bounds at small and large times, while the Euclidean
interaction kernel is not homogeneous under the Grushin
dilations. We therefore obtain the solvability conditions by
matching the semigroup and convolution estimates directly.

The principal nonlinear contributions of the present work concern both the
critical and supercritical integrability regimes. At the endpoint of the
short-time nonlinear time-integrability estimate, the pointwise smoothing
bound produces a nonintegrable time singularity. Lorentz duality, combined
with a Grushin Yamazaki-type estimate and the mixed semigroup factorization,
provides a uniform bound for the nonlinear Duhamel operator and yields global
mild solutions for sufficiently small initial data in the critical
Marcinkiewicz space. Beyond this endpoint, although the direct fixed-point
argument fails in the original space \(L^{p,\infty}\), the two-scale structure
of the mixed diffusion allows well-posedness to be recovered by imposing
higher spatial integrability on the initial data. In this way, the local and
global theory extends across the subcritical, critical, and supercritical
regimes relative to the prescribed Marcinkiewicz space.

Our argument begins with the linear evolution

$$
S(t)=e^{-t(G+G^\delta)}.
$$

Subordination transfers the Grushin heat estimates to the fractional
semigroup, and the commuting factorization

$$
S(t)=e^{-tG}e^{-tG^\delta}
$$

then combines the two diffusion bounds. Interpolation gives, in particular,
for \(1<a\leq b<\infty\) and \(1\leq s\leq\infty\),

$$
\|S(t)f\|_{L^{b,s}}
\leq
C\min\left\{
t^{-\frac Q2(\frac1a-\frac1b)},
t^{-\frac Q{2\delta}(\frac1a-\frac1b)}
\right\}
\|f\|_{L^{a,s}},
\qquad t>0.
$$

This estimate records the order-two smoothing bound at short times and the
fractional decay bound at large times. It also avoids requiring an explicit
formula for the fractional Grushin heat kernel.

To describe the nonlinear results, let \(1<\rho<p<\infty\), and define \(r\)
by

$$
\frac1r=\frac{\rho}{p}+\frac{\alpha}{d}-1.
$$

In the admissible range \(1<r\leq p\), the Lorentz-space H\"older and
convolution inequalities yield

$$
\|I_\alpha(|u|^\rho)\|_{L^{r,\infty}}
\leq C\|u\|_{L^{p,\infty}}^\rho.
$$

We apply these estimates to the mild formulation

$$
u(t)=S(t)u_0+
\int_0^t S(t-s)I_\alpha(|u(s)|^\rho)\,ds.
$$

Combining the potential estimate with the mixed semigroup bound shows that
the short-time singularity in the nonlinear term is governed by

$$
\theta=
\frac Q2\left(
\frac{\rho-1}{p}+\frac{\alpha}{d}-1
\right).
$$

Writing

$$
\frac1\beta=\frac2Q+1-\frac{\alpha}{d},
$$

the condition \(\theta<1\) is equivalent to

$$
\rho<1+\frac{p}{\beta}.
$$

Under these assumptions, for every
\(u_0\in L^{p,\infty}(\mathbb R^d)\), we establish local existence and
uniqueness of mild solutions in

$$
L^\infty\bigl((0,T);L^{p,\infty}(\mathbb R^d)\bigr),
$$

together with continuous dependence on the initial data and a blow-up
alternative.

At the endpoint

$$
\rho=1+\frac{p}{\beta},
$$

the direct estimate produces the nonintegrable time factor
\((t-s)^{-1}\). We therefore use Lorentz duality and a Yamazaki-type
integral estimate to obtain a uniform bound for the nonlinear Duhamel
operator. This yields a global mild solution for sufficiently small
\(u_0\in L^{p,\infty}(\mathbb R^d)\), unique in the corresponding small
ball of

$$
L^\infty\bigl((0,\infty);L^{p,\infty}(\mathbb R^d)\bigr).
$$

Here the endpoint refers to the time-integrability balance above and should
not be interpreted as a sharp Fujita threshold, whose identification would
also require corresponding nonexistence results.

We further treat the regime

$$
\rho>1+\frac{p}{\beta},
$$

which is supercritical relative to the prescribed space
\(L^{p,\infty}(\mathbb R^d)\). In this case the short-time Duhamel
singularity has exponent greater than one, and hence the preceding
\(L^{p,\infty}\)-based contraction argument no longer applies. The mixed
local--nonlocal diffusion, however, provides an additional range of
integrability. Define

$$
\frac1{\beta_\delta}
=
\frac{2\delta}{Q}+1-\frac{\alpha}{d}.
$$

Since \(0<\delta<1\), one has \(\beta_\delta>\beta\). We show that for every

$$
\beta(\rho-1)<q<\beta_\delta(\rho-1),
$$

the problem is locally well posed for initial data
\(u_0\in L^{q,\infty}(\mathbb R^d)\), and the corresponding solution is
global whenever the initial norm is sufficiently small. The lower bound on
\(q\) restores integrability of the short-time singularity, while the upper
bound exploits the fractional large-time decay of the mixed semigroup.
Thus the supercritical nonlinearity can be treated by passing from the
original space \(L^{p,\infty}\) to a higher-integrability Marcinkiewicz
space \(L^{q,\infty}\). This also makes explicit the distinct roles played
by the local and fractional components of the diffusion in the nonlinear
theory.

The paper is organized as follows. Section~2 recalls the Lorentz and
Marcinkiewicz spaces and the inequalities needed for the nonlinear analysis.
Section~3 establishes the fractional and mixed Grushin semigroup estimates.
Section~4 proves local well-posedness and the blow-up alternative in the
subcritical regime. Section~5 establishes global small-data well-posedness
at the critical endpoint and then treats the supercritical regime in
higher-integrability Marcinkiewicz spaces.

\section{Preliminaries}

We collect the properties of Lorentz and Marcinkiewicz spaces needed for the
semigroup estimates and the fixed-point argument. These spaces allow us to
treat singular data beyond the Lebesgue scale and to control the nonlinear
term through suitable product and convolution estimates. Throughout this
section, $(X,\mu)$ is a $\sigma$-finite measure space, and measurable functions
are identified whenever they agree almost everywhere. We omit $(X,\mu)$ from
the notation when the underlying measure space is clear. For a detailed
account of Lorentz spaces, we refer to \cite{MR3243734}.

\begin{definition}
Let $f$ be a measurable function on $(X,\mu)$. Its distribution function is
defined by
\[
    d_f(\alpha)
    :=\mu\bigl(\{x\in X:|f(x)|>\alpha\}\bigr),
    \qquad \alpha\geq0.
\]
The decreasing rearrangement of $f$ is the nonincreasing function
\[
    f^*(t):=\inf\{\alpha\geq0:d_f(\alpha)\leq t\},
    \qquad t>0,
\]
where $\inf\varnothing=\infty$. We also define the maximal rearrangement by
\[
    f^{**}(t):=\frac1t\int_0^t f^*(s)\,\mathrm{d}s,
    \qquad t>0.
\]
\end{definition}

\begin{definition}
Let $0<p<\infty$ and $0<q\leq\infty$. The Lorentz space
$L^{p,q}(X,\mu)$ consists of all measurable functions $f$ for which the
quasi-norm
\[
    \|f\|_{L^{p,q}}^*
    :=
    \begin{cases}
        \displaystyle
        \left(\int_0^\infty
        \bigl[t^{1/p}f^*(t)\bigr]^q
        \frac{\mathrm{d}t}{t}\right)^{1/q},
        & 0<q<\infty,\\[2ex]
        \displaystyle
        \sup_{t>0}t^{1/p}f^*(t),
        & q=\infty,
    \end{cases}
\]
is finite. At the remaining endpoint, we set
$L^{\infty,\infty}(X,\mu)=L^\infty(X,\mu)$, equipped with the usual
essential-supremum norm.
\end{definition}

With this normalization, $L^{p,p}=L^p$ and
$\|f\|_{L^{p,p}}^*=\|f\|_{L^p}$. For $q=\infty$, the relation between
$d_f$ and $f^*$ gives
\begin{align*}
    \|f\|_{L^{p,\infty}}^*
    &=\sup_{\alpha>0}\alpha\,d_f(\alpha)^{1/p}\\
    &=\inf\left\{C>0:
        d_f(\alpha)\leq\left(\frac{C}{\alpha}\right)^p
        \text{ for every }\alpha>0\right\}.
\end{align*}
Thus $L^{p,\infty}$ coincides with the weak $L^p$ space, also called the
Marcinkiewicz space. Chebyshev's inequality yields
\[
    \|f\|_{L^{p,\infty}}^*\leq\|f\|_{L^p},
    \qquad f\in L^p,
\]
and hence the continuous inclusion $L^p\hookrightarrow L^{p,\infty}$.

\begin{remark}
The inclusion is generally strict. On $\mathbb{R}^d$, the function
$h(x)=|x|^{-d/p}$, defined arbitrarily at the origin, satisfies
\[
    d_h(\alpha)=\nu_d\alpha^{-p},
    \qquad \alpha>0,
\]
where $\nu_d$ denotes the measure of the unit ball. Consequently,
\[
    h\in L^{p,\infty}(\mathbb{R}^d)\setminus L^p(\mathbb{R}^d),
    \qquad
    \|h\|_{L^{p,\infty}}^*=\nu_d^{1/p}.
\]
This example illustrates the role of Marcinkiewicz spaces in admitting
singular profiles with borderline integrability.
\end{remark}

For the Banach-space arguments below, we use an equivalent norm constructed
from $f^{**}$. If $1<p<\infty$ and $1\leq q\leq\infty$, set
\[
    \|f\|_{L^{p,q}}
    :=
    \begin{cases}
        \displaystyle
        \left(\int_0^\infty
        \bigl[t^{1/p}f^{**}(t)\bigr]^q
        \frac{\mathrm{d}t}{t}\right)^{1/q},
        & 1\leq q<\infty,\\[2ex]
        \displaystyle
        \sup_{t>0}t^{1/p}f^{**}(t),
        & q=\infty.
    \end{cases}
\]
The absence of a superscript $*$ will always indicate this norm in the
stated range of indices.

\begin{proposition}
Let $1<p<\infty$ and $1\leq q\leq\infty$. Then
$L^{p,q}(X,\mu)$, equipped with $\|\cdot\|_{L^{p,q}}$, is a Banach space,
and
\[
    \|f\|_{L^{p,q}}^*
    \leq\|f\|_{L^{p,q}}
    \leq p'\|f\|_{L^{p,q}}^*,
    \qquad p':=\frac{p}{p-1}.
\]
\end{proposition}

The first inequality follows from $f^*\leq f^{**}$, and the second is a
consequence of Hardy's inequality. In particular, the two functionals
define the same topology, but only the unstarred functional will be used
as a norm in the contraction argument.

\begin{remark}
\label{remark2}
Let $0<p,\rho<\infty$ and $0<q\leq\infty$, with the convention
$\rho\infty=\infty$. Since
$(|g|^\rho)^*=(g^*)^\rho$, one has the exact identity
\[
    \big\||g|^\rho\big\|_{L^{p,q}}^*
    =\left(\|g\|_{L^{p\rho,q\rho}}^*\right)^\rho.
\]
If, in addition, $1<p,p\rho<\infty$ and
$1\leq q,q\rho\leq\infty$, norm equivalence gives
\begin{equation}
\label{prhorelation}
    \big\||g|^\rho\big\|_{L^{p,q}}
    \leq p'\left(\|g\|_{L^{p\rho,q\rho}}\right)^\rho.
\end{equation}
\end{remark}

We next recall the duality properties used in the endpoint estimates. For
a Lorentz space $E=L^{p,q}$, its associate space, denoted by $E^\times$,
consists of all measurable functions $g$ such that
\[
    \|g\|_{E^\times}
    :=\sup_{\|f\|_{L^{p,q}}^*\leq1}
      \int_X|fg|\,\mathrm{d}\mu<\infty.
\]
In the Banach range, replacing the quasi-norm in this definition by the
equivalent norm gives the same associate space with an equivalent norm.

\begin{proposition}
Let $1<p<\infty$. The associate spaces satisfy
\[
    (L^{p,q})^\times
    =
    \begin{cases}
        L^{p',\infty}, & 0<q\leq1,\\
        L^{p',q'}, & 1<q\leq\infty,
    \end{cases}
\]
with equivalence of norms, where $1/q+1/q'=1$ and $\infty'=1$.
For $0<q<\infty$, every continuous linear functional $\psi$ on $L^{p,q}$
has the representation
\[
    \psi(f)=\int_X f g\,\mathrm{d}\mu
\]
for a unique $g\in(L^{p,q})^\times$, up to equality almost everywhere.
In particular,
\[
    (L^{p',1})^*=L^{p,\infty},
    \qquad 1<p<\infty,
\]
under the integral pairing and with equivalent norms.
\end{proposition}

Here $E^*$ denotes the continuous dual. For $q=\infty$, the associate
space $L^{p',1}$ need not exhaust the full continuous dual of
$L^{p,\infty}$. The preceding predual identification nevertheless yields
the norm characterization needed below:
\[
    \|f\|_{L^{p,\infty}}
    \asymp
    \sup_{\|\varphi\|_{L^{p',1}}\leq1}
    \left|\int_X f\varphi\,\mathrm{d}\mu\right|.
\]
Throughout, $A\lesssim B$ means $A\leq CB$ for a constant independent of
the functions being estimated, and $A\asymp B$ means that both inequalities
hold.

We shall also use H\"older's inequality in Lorentz spaces; see
\cite{MR223874}. Related formulations appear in
\cite{MR146673,MR2279332,MR5014751}.

\begin{proposition}[H\"older's inequality]
\label{generalisedholder}
Let $1<p_1,p_2<\infty$ and $1\leq q_1,q_2,s\leq\infty$. Suppose that
\[
    \frac1r=\frac1{p_1}+\frac1{p_2}<1,
    \qquad
    \frac1s\leq\frac1{q_1}+\frac1{q_2},
\]
where $1/\infty=0$. If $f\in L^{p_1,q_1}$ and $g\in L^{p_2,q_2}$,
then $fg\in L^{r,s}$ and
\[
    \|fg\|_{L^{r,s}}
    \leq C\|f\|_{L^{p_1,q_1}}\|g\|_{L^{p_2,q_2}},
\]
where $C$ depends only on the indices. At the conjugate-exponent endpoint,
for $1<p<\infty$ and $1\leq q\leq\infty$, one also has
\[
    \int_X|fg|\,\mathrm{d}\mu
    \leq C\|f\|_{L^{p,q}}\|g\|_{L^{p',q'}},
    \qquad \frac1q+\frac1{q'}=1.
\]
\end{proposition}

For convolution estimates, the underlying space is
$\mathbb{R}^d$ with Lebesgue measure. The following Lorentz-space version
of Young's inequality is due to O'Neil; see \cite{MR146673}.

\begin{theorem}[O'Neil's convolution inequality]
\label{hardysobolevinequlity}
Let $1<p_1,p_2,p<\infty$ and $0<q_1,q_2,q\leq\infty$ satisfy
\[
    1+\frac1p=\frac1{p_1}+\frac1{p_2},
    \qquad
    \frac1q\leq\frac1{q_1}+\frac1{q_2}.
\]
If $f\in L^{p_1,q_1}(\mathbb{R}^d)$ and
$g\in L^{p_2,q_2}(\mathbb{R}^d)$, then
$f*g\in L^{p,q}(\mathbb{R}^d)$ and
\[
    \|f*g\|_{L^{p,q}}^*
    \leq C\|f\|_{L^{p_1,q_1}}^*\|g\|_{L^{p_2,q_2}}^*,
\]
where $C$ depends only on the indices. If $q_1,q_2,q\geq1$, the same
estimate holds with the corresponding unstarred norms.
\end{theorem}

Finally, we record the local Lipschitz estimate for the power nonlinearity
that enters the fixed-point argument.

\begin{lemma}
\label{fractionallemma}
Let $1<\rho<p<\infty$ and $u,v\in L^{p,\infty}(X,\mu)$. For either
$F(w)=|w|^\rho$ or $F(w)=|w|^{\rho-1}w$, one has
\[
    \|F(u)-F(v)\|_{L^{p/\rho,\infty}}
    \leq C_{p,\rho}\|u-v\|_{L^{p,\infty}}
    \left(
        \|u\|_{L^{p,\infty}}^{\rho-1}
        +\|v\|_{L^{p,\infty}}^{\rho-1}
    \right).
\]
In particular, each of these nonlinear maps is locally Lipschitz from
$L^{p,\infty}$ into $L^{p/\rho,\infty}$.
\end{lemma}

\begin{proof}
For either choice of $F$, the pointwise inequality
\[
    |F(a)-F(b)|
    \leq C_\rho\bigl(|a|^{\rho-1}+|b|^{\rho-1}\bigr)|a-b|
\]
holds for real or complex $a,b$. Set $r=p/\rho$ and
$\eta=p/(\rho-1)$. Since $r,\eta>1$ and
$1/r=1/p+1/\eta$, Proposition~\ref{generalisedholder} gives
\begin{align*}
    \|F(u)-F(v)\|_{L^{r,\infty}}
    &\leq C_{p,\rho}\|u-v\|_{L^{p,\infty}}
    \left(
        \big\||u|^{\rho-1}\big\|_{L^{\eta,\infty}}
        +\big\||v|^{\rho-1}\big\|_{L^{\eta,\infty}}
    \right).
\end{align*}
Applying \eqref{prhorelation} with exponent $\rho-1$ and using
$\eta(\rho-1)=p$ proves the estimate.
\end{proof}

\section{Local-nonlocal Grushin operator and its semigroup}

In this section, we develop the linear estimates underlying the solvability
theory for the nonlinear problem. We begin with the classical Grushin heat
semigroup and its representation through the heat kernels of scaled Hermite
operators. We then use the spectral calculus and subordination to study the
fractional Grushin semigroup. Combining these two families of estimates
yields bounds for the mixed local-nonlocal semigroup in Lebesgue and Lorentz
spaces, which will be used to control the Duhamel term in the subsequent
fixed-point arguments.

Recall that the Grushin operator is defined by
\[
    \Delta_G=-\frac12\bigl(\Delta_x+|x|^2\Delta_y\bigr),
    \qquad (x,y)\in\mathbb{R}^N\times\mathbb{R}^k,
\]
where $\Delta_x$ and $\Delta_y$ are the classical Laplacians in the indicated
variables. Introduced in \cite{MR279436}, this operator is a basic model of
degenerate diffusion, with degeneracy along $\{x=0\}$. We denote by $G$ its
nonnegative self-adjoint realization on $L^2(\mathbb{R}^{N+k})$. For
$0<\delta<1$, the mixed operator and its heat semigroup are given by
\[
    \mathcal{G}=G+G^\delta,
    \qquad
    S_{\mathcal{G}}(t)=e^{-t\mathcal{G}}
    =e^{-tG}e^{-tG^\delta},
    \qquad t\geq0,
\]
where $G^\delta$ is defined by the spectral calculus. The factorization
follows because both operators are functions of the same self-adjoint
operator $G$.

We first describe the kernel of $S_G(t)=e^{-tG}$. We use the partial Fourier
transform convention
\[
    \widehat{\varphi}(x,\xi)
    =\int_{\mathbb{R}^k}e^{-iy\cdot\xi}\varphi(x,y)\,\mathrm{d}y,
    \qquad
    \varphi(x,y)
    =\frac1{(2\pi)^k}
      \int_{\mathbb{R}^k}e^{iy\cdot\xi}
      \widehat{\varphi}(x,\xi)\,\mathrm{d}\xi.
\]
For $\varphi\in\mathcal{S}(\mathbb{R}^{N+k})$, this transform gives
\[
    \widehat{\Delta_G\varphi}(x,\xi)
    =H_{|\xi|}\widehat{\varphi}(x,\xi),
    \qquad
    H_\lambda:=\frac12\bigl(-\Delta_x+\lambda^2|x|^2\bigr).
\]
Thus the Grushin heat equation $\partial_tu+Gu=0$ reduces, at each Fourier
frequency, to the heat equation for a scaled Hermite operator. For
$\lambda>0$, the corresponding fundamental solution satisfies
\begin{equation}
\label{heateq}
    \partial_t v+H_\lambda v=0,
    \qquad v(\cdot,0)=\delta_{x_0}.
\end{equation}
Equivalently,
$\partial_t v=\frac12\Delta_xv-\frac12\lambda^2|x|^2v$.

The kernel of $e^{-tH_\lambda}$ is given by Mehler's formula; see
\cite{MR5073613} and the geometric approach developed in \cite{MR2723056}:
\begin{equation}
\label{heatkernelL}
\begin{aligned}
    L_\lambda(x,x_0;t)
    &=\left(\frac{\lambda}{2\pi\sinh(\lambda t)}\right)^{N/2}\\
    &\quad\times\exp\!\left\{
        -\frac{\lambda}{2\sinh(\lambda t)}
        \left[(|x|^2+|x_0|^2)\cosh(\lambda t)-2x\cdot x_0\right]
    \right\},
    \qquad t>0.
\end{aligned}
\end{equation}
At $\lambda=0$, this expression is understood by continuity and reduces to
the Euclidean heat kernel
\[
    L_0(x,x_0;t)
    =(2\pi t)^{-N/2}
      \exp\!\left(-\frac{|x-x_0|^2}{2t}\right).
\]

Taking the inverse partial Fourier transform in \eqref{heatkernelL}, with
$\lambda=|\xi|$, yields the Grushin heat kernel
\[
    K(x,x_0,y;t)
    =\frac1{(2\pi)^k}
      \int_{\mathbb{R}^k}
      e^{i\xi\cdot y}L_{|\xi|}(x,x_0;t)\,\mathrm{d}\xi.
\]
More explicitly,
\begin{align*}
    K(x,x_0,y;t)
    &=\frac1{(2\pi)^{N/2+k}}
      \int_{\mathbb{R}^k}
      \left(\frac{|\xi|}{\sinh(|\xi|t)}\right)^{N/2}
      e^{i\xi\cdot y}\\
    &\quad\times\exp\!\left\{
        -\frac{|\xi|}{2}
        \left[
            (|x|^2+|x_0|^2)\coth(|\xi|t)
            -2x\cdot x_0\,\operatorname{csch}(|\xi|t)
        \right]
      \right\}\,\mathrm{d}\xi,
\end{align*}
for $(x,x_0,y)\in\mathbb{R}^{2N+k}$ and $t>0$. The integrand at $\xi=0$
is interpreted through the preceding limit. This representation also
appears in \cite{MR4625004}, after accounting for the operator and Fourier
normalizations used there.

The associated heat semigroup therefore admits the integral representation
\[
    S_G(t)\varphi(x,y)
    =\int_{\mathbb{R}^{N+k}}
      K(x,w,y-z;t)\varphi(w,z)\,\mathrm{d}w\,\mathrm{d}z,
    \qquad t>0,
\]
initially for $\varphi\in\mathcal{S}(\mathbb{R}^{N+k})$ and, by extension,
on the Lebesgue spaces considered below. The kernel depends separately on
$x$ and $w$, reflecting the variable coefficients of $G$, whereas its
dependence on the $y$-variables is through the difference $y-z$.

We now recall the $L^p$--$L^q$ estimate for $S_G(t)$ established in
\cite{MR5073613}. Its decay rate is governed by the homogeneous dimension
$Q=N+2k$ associated with the Grushin dilations
$(x,y)\mapsto(rx,r^2y)$.

\begin{theorem}
The family $\{S_G(t)\}_{t\geq0}$ defines a semigroup of positive
contractions on $L^p(\mathbb{R}^{N+k})$ for every $1\leq p\leq\infty$.
For $1\leq p<\infty$, this semigroup is strongly continuous on
$[0,\infty)$.

Moreover, if $1\leq p\leq r\leq\infty$, then
\begin{equation}
\label{Grusinestimates}
    \|S_G(t)\varphi\|_{L^r(\mathbb{R}^{N+k})}
    \leq C\,t^{-\frac{N+2k}{2}
        \left(\frac1p-\frac1r\right)}
    \|\varphi\|_{L^p(\mathbb{R}^{N+k})},
    \qquad t>0,
\end{equation}
where $C$ depends only on $N,k,p,r$. For every
$\varphi\in L^p(\mathbb{R}^{N+k})$, $1\leq p\leq\infty$, and every
$t_0>0$,
\[
    \lim_{t\to t_0}
    \|S_G(t)\varphi-S_G(t_0)\varphi\|_{L^p(\mathbb{R}^{N+k})}=0.
\]
When $p<\infty$, this convergence also holds at $t_0=0$, with
$S_G(0)=I$.
\end{theorem}

\begin{remark}
Strong continuity at $t=0$ does not hold on all of
$L^\infty(\mathbb{R}^{N+k})$. At this endpoint, one has instead
\[
    S_G(t)\varphi\overset{*}{\rightharpoonup}\varphi
    \quad\text{in }L^\infty(\mathbb{R}^{N+k})
    \quad\text{as }t\downarrow0.
\]
Indeed, by symmetry and strong continuity on $L^1$,
\begin{align*}
    \int_{\mathbb{R}^{N+k}}
    \bigl(S_G(t)\varphi-\varphi\bigr)\psi\,\mathrm{d}x\,\mathrm{d}y
    &=\int_{\mathbb{R}^{N+k}}
    \varphi\bigl(S_G(t)\psi-\psi\bigr)\,\mathrm{d}x\,\mathrm{d}y\\
    &\longrightarrow0
\end{align*}
for every $\psi\in L^1(\mathbb{R}^{N+k})$.
\end{remark}

We next describe $G$ through its spectral resolution. This description
connects the heat-kernel representation with the functional calculus and
provides the basis for defining fractional powers. The nonnegative
self-adjoint realization of the Grushin operator and its spectral
transform are studied in \cite{MR5099478}; here we retain the factor
$1/2$ in the definition of $\Delta_G$.

Consistently with the Fourier convention fixed above, write
\[
    f^\lambda(x):=\widehat f(x,\lambda)
    =\int_{\mathbb{R}^k}
      f(x,y)e^{-i\lambda\cdot y}\,\mathrm{d}y.
\]
This transform is initially defined on Schwartz class functions and extended
to $L^2$ by Plancherel's theorem.
For $f\in\mathcal{S}(\mathbb{R}^{N+k})$, one has
\[
    (Gf)^\lambda=H_{|\lambda|}f^\lambda,
    \qquad
    H_{|\lambda|}
    =\frac12\bigl(-\Delta_x+|\lambda|^2|x|^2\bigr),
\]
and hence
\[
    Gf(x,y)
    =\frac1{(2\pi)^k}
      \int_{\mathbb{R}^k}
      e^{i\lambda\cdot y}
      \bigl(H_{|\lambda|}f^\lambda\bigr)(x)\,\mathrm{d}\lambda.
\]

For $\lambda\neq0$, the scaled Hermite operator $H_{|\lambda|}$ has
eigenvalues
\[
    \mu_j(\lambda)=\frac{(2j+N)|\lambda|}{2},
    \qquad j=0,1,2,\ldots.
\]
Let $P_j(\lambda)$ denote the orthogonal projection of
$L^2(\mathbb{R}^N)$ onto the corresponding eigenspace. Its spectral
resolution is
\[
    H_{|\lambda|}
    =\sum_{j=0}^\infty\mu_j(\lambda)P_j(\lambda),
\]
with convergence on the operator domain in $L^2(\mathbb{R}^N)$.
Consequently, for $f\in D(G)$,
\[
    Gf(x,y)
    =\frac1{(2\pi)^k}
      \int_{\mathbb{R}^k}e^{i\lambda\cdot y}
      \sum_{j=0}^\infty
      \frac{(2j+N)|\lambda|}{2}
      \bigl(P_j(\lambda)f^\lambda\bigr)(x)\,\mathrm{d}\lambda.
\]
This identity is understood in $L^2(\mathbb{R}^{N+k})$, with the integral
interpreted as an inverse partial Fourier transform. The single frequency
$\lambda=0$ does not affect this representation.

Applying the spectral calculus to the function $s\mapsto e^{-ts}$ gives
\[
    \bigl(S_G(t)f\bigr)^\lambda
    =e^{-tH_{|\lambda|}}f^\lambda
    =\sum_{j=0}^\infty
      e^{-\frac{t}{2}(2j+N)|\lambda|}
      P_j(\lambda)f^\lambda.
\]
Thus, for $f\in L^2(\mathbb{R}^{N+k})$ and $t>0$,
\begin{align*}
    S_G(t)f(x,y)=e^{-tG}f(x,y)
    &=\frac1{(2\pi)^k}
      \int_{\mathbb{R}^k}e^{i\lambda\cdot y}
      \sum_{j=0}^\infty
      e^{-\frac{t}{2}(2j+N)|\lambda|}\\
    &\qquad\qquad\times
      \bigl(P_j(\lambda)f^\lambda\bigr)(x)\,\mathrm{d}\lambda,
\end{align*}
again in the $L^2$ sense. This spectral formula agrees with the
heat-kernel representation obtained above and will be used to construct
the fractional Grushin semigroup.
We now define the fractional component of the mixed operator
\[
    \mathcal{G}=G+G^\delta,
    \qquad 0<\delta<1.
\]
Since $G$ is nonnegative and self-adjoint, its fractional power $G^\delta$
is defined by applying the spectral calculus to the function
$\sigma\mapsto\sigma^\delta$. In terms of the preceding spectral
resolution, for $f\in D(G^\delta)$,
\[
    G^\delta f(x,y)
    =\frac1{(2\pi)^k}
      \int_{\mathbb{R}^k}e^{i\lambda\cdot y}
      \sum_{j=0}^\infty
      \left(\frac{(2j+N)|\lambda|}{2}\right)^\delta
      \bigl(P_j(\lambda)f^\lambda\bigr)(x)\,\mathrm{d}\lambda.
\]
This identity is understood in $L^2(\mathbb{R}^{N+k})$. We refer to
Balhara~\cite{MR3938515} for the study of fractional powers of Grushin
operators and associated Hardy inequalities.

The heat semigroup associated with $G^\delta$ can be obtained from
$S_G(t)$ by Bochner subordination, see, for example, \cite{MR2039954}.
To recall the general formulation, let
$\varphi:[0,\infty)\to[0,\infty)$ be a Bernstein function satisfying
$\varphi(0)=0$. Then there exists a   semigroup of probability
measures $\{\nu_t\}_{t\geq0}$ on $[0,\infty)$, characterized by
\begin{equation}
\label{subordinator}
    e^{-t\varphi(\sigma)}
    =\int_{[0,\infty)}e^{-s\sigma}\,\nu_t(\mathrm{d}s),
    \qquad t>0,\quad\sigma\geq0.
\end{equation}
The spectral theorem therefore yields
\[
    e^{-t\varphi(G)}f
    =\int_{[0,\infty)}S_G(s)f\,\nu_t(\mathrm{d}s),
    \qquad f\in L^2(\mathbb{R}^{N+k}),
\]
where the integral is understood in the Bochner sense. If $\nu_t$ is
absolutely continuous, with density $\eta_t$, this becomes
\[
    e^{-t\varphi(G)}f
    =\int_0^\infty\eta_t(s)S_G(s)f\,\mathrm{d}s.
\]

For $0<\delta<1$, the function $\varphi(\sigma)=\sigma^\delta$ is a
Bernstein function. Its associated measures have nonnegative continuous
densities $\eta_t^{(\delta)}$ on $(0,\infty)$, namely the densities of
the $\delta$-stable subordinator, satisfying
\[
    \int_0^\infty
    \eta_t^{(\delta)}(s)e^{-s\sigma}\,\mathrm{d}s
    =e^{-t\sigma^\delta},
    \qquad \sigma\geq0.
\]
In particular,
\[
    \int_0^\infty\eta_t^{(\delta)}(s)\,\mathrm{d}s=1,
    \qquad t>0.
\]
Consequently, the fractional Grushin heat semigroup admits the
representation
\begin{equation}
\label{gfrac}
    e^{-tG^\delta}f
    =\int_0^\infty
      \eta_t^{(\delta)}(s)S_G(s)f\,\mathrm{d}s,
    \qquad t>0.
\end{equation}
Besides its $L^2$ interpretation, this formula holds as a Bochner integral
in $L^p$ for $1\leq p<\infty$, and in the weak-$*$ sense for $p=\infty$.

Finally, the commuting factorization of the mixed semigroup gives
\[
    e^{-t\mathcal{G}}f
    =e^{-tG}e^{-tG^\delta}f
    =\int_0^\infty
      \eta_t^{(\delta)}(s)S_G(t+s)f\,\mathrm{d}s.
\]
These representations allow the smoothing estimates for $S_G(t)$ to be
transferred to the fractional and mixed Grushin semigroups.

\subsection{$L^p$--$L^q$ estimates for the fractional Grushin semigroup}

The subordination formula transfers the smoothing properties of the
Grushin heat semigroup to its fractional counterpart. The dependence on
time is determined by the negative moments of the stable subordinator,
which can be evaluated explicitly using the Gamma function.

\begin{proposition}
Let $0<\delta<1$ and $1\leq p\leq q\leq\infty$. There exists a constant
$C=C(N,k,p,q,\delta)>0$ such that
\begin{equation}
\label{fractionalgrushinestimate}
    \|e^{-tG^\delta}f\|_{L^q(\mathbb{R}^{N+k})}
    \leq C\,t^{-\frac{N+2k}{2\delta}
        \left(\frac1p-\frac1q\right)}
    \|f\|_{L^p(\mathbb{R}^{N+k})}
\end{equation}
for every $t>0$ and $f\in L^p(\mathbb{R}^{N+k})$. When $p=q$, one may
take $C=1$.
\end{proposition}

\begin{proof}
Suppose first that $p=q$. Since $S_G(s)$ is contractive on $L^p$ and
$\eta_t^{(\delta)}$ has total mass one, \eqref{gfrac} gives
\[
    \|e^{-tG^\delta}f\|_{L^p}
    \leq\int_0^\infty
      \eta_t^{(\delta)}(s)\|S_G(s)f\|_{L^p}\,\mathrm{d}s
    \leq\|f\|_{L^p}.
\]

Now let $p<q$ and set
\[
    a:=\frac{N+2k}{2}\left(\frac1p-\frac1q\right)>0.
\]
Combining \eqref{gfrac} with \eqref{Grusinestimates}, we obtain
\begin{align}
\label{abc}
    \|e^{-tG^\delta}f\|_{L^q}
    &\leq\int_0^\infty
      \eta_t^{(\delta)}(s)\|S_G(s)f\|_{L^q}\,\mathrm{d}s
      \nonumber\\
    &\leq C\|f\|_{L^p}
      \int_0^\infty s^{-a}\eta_t^{(\delta)}(s)\,\mathrm{d}s.
\end{align}
The integral norm inequalities follow from Minkowski's inequality for
finite target exponents and from duality with $L^1$ at the
$L^\infty$ endpoint.

Recall that
\[
    \Gamma(a)=\int_0^\infty u^{a-1}e^{-u}\,\mathrm{d}u,
    \qquad a>0.
\]
A change of variables yields
\begin{equation}
\label{gammaz}
    s^{-a}
    =\frac1{\Gamma(a)}
      \int_0^\infty u^{a-1}e^{-su}\,\mathrm{d}u,
    \qquad s>0.
\end{equation}
Using Tonelli's theorem and the Laplace-transform identity for
$\eta_t^{(\delta)}$, we find
\begin{align*}
    \int_0^\infty s^{-a}\eta_t^{(\delta)}(s)\,\mathrm{d}s
    &=\frac1{\Gamma(a)}
      \int_0^\infty\int_0^\infty
      u^{a-1}e^{-su}\eta_t^{(\delta)}(s)
      \,\mathrm{d}u\,\mathrm{d}s\\
    &=\frac1{\Gamma(a)}
      \int_0^\infty u^{a-1}
      \left(\int_0^\infty
      e^{-su}\eta_t^{(\delta)}(s)\,\mathrm{d}s\right)
      \,\mathrm{d}u\\
    &=\frac1{\Gamma(a)}
      \int_0^\infty u^{a-1}e^{-tu^\delta}\,\mathrm{d}u.
\end{align*}
The substitution $v=tu^\delta$ therefore gives
\begin{align*}
    \int_0^\infty s^{-a}\eta_t^{(\delta)}(s)\,\mathrm{d}s
    &=\frac{t^{-a/\delta}}{\delta\Gamma(a)}
      \int_0^\infty v^{a/\delta-1}e^{-v}\,\mathrm{d}v\\
    &=\frac{\Gamma(a/\delta)}{\delta\Gamma(a)}
      t^{-a/\delta}.
\end{align*}
Substituting this identity into \eqref{abc} and recalling the definition
of $a$ proves \eqref{fractionalgrushinestimate}.
\end{proof}

\subsection{Estimates for the mixed local-nonlocal Grushin operator}

We now combine the estimates for the classical and fractional Grushin
semigroups. Since $G$ and $G^\delta$ are defined through the same spectral
resolution, the heat semigroup associated with
$\mathcal{G}=G+G^\delta$ satisfies
\[
    e^{-t\mathcal{G}}=e^{-tG}e^{-tG^\delta}
    =e^{-tG^\delta}e^{-tG}.
\]
This factorization allows either component to provide the smoothing
estimate while the other acts contractively on the initial-data space.

\begin{theorem}
\label{LpLrestimate}
Let $0<\delta<1$. The family $\{e^{-t\mathcal{G}}\}_{t\geq0}$ is a
semigroup of positive contractions on $L^p(\mathbb{R}^{N+k})$ for every
$1\leq p\leq\infty$. It is strongly continuous on $L^p$ for
$1\leq p<\infty$ and weak-$*$ continuous on $L^\infty$ with respect to
its predual $L^1$.

If $1\leq p\leq r\leq\infty$, then
\[
    \|e^{-t\mathcal{G}}f\|_{L^r(\mathbb{R}^{N+k})}
    \leq C\,t_\delta^{-\frac{N+2k}{2}
        \left(\frac1p-\frac1r\right)}
    \|f\|_{L^p(\mathbb{R}^{N+k})},
    \qquad t>0,
\]
where
\[
    t_\delta:=\max\{t,t^{1/\delta}\},
\]
and $C$ depends only on $N,k,p,r,\delta$. When $p=r$, one may take
$C=1$.
\end{theorem}

\begin{proof}
Set
\[
    a:=\frac{N+2k}{2}\left(\frac1p-\frac1r\right)\geq0.
\]
Using \eqref{Grusinestimates} and the $L^p$-contractivity of
$e^{-tG^\delta}$, we obtain
\begin{align*}
    \|e^{-t\mathcal{G}}f\|_{L^r}
    &=\|S_G(t)e^{-tG^\delta}f\|_{L^r}\\
    &\leq C t^{-a}\|e^{-tG^\delta}f\|_{L^p}
    \leq C t^{-a}\|f\|_{L^p}.
\end{align*}
Reversing the order of the commuting factors and applying
\eqref{fractionalgrushinestimate} gives
\begin{align*}
    \|e^{-t\mathcal{G}}f\|_{L^r}
    &=\|e^{-tG^\delta}S_G(t)f\|_{L^r}\\
    &\leq C t^{-a/\delta}\|S_G(t)f\|_{L^p}
    \leq C t^{-a/\delta}\|f\|_{L^p}.
\end{align*}
Combining these bounds yields
\[
    \|e^{-t\mathcal{G}}f\|_{L^r}
    \leq C\min\{t^{-a},t^{-a/\delta}\}\|f\|_{L^p}
    =C t_\delta^{-a}\|f\|_{L^p}.
\]
Positivity and contractivity follow from the corresponding properties
of the two factors. For $p<\infty$, Bochner subordination preserves
strong continuity, and hence
\begin{align*}
    \|e^{-t\mathcal{G}}f-f\|_{L^p}
    &\leq\|e^{-tG^\delta}f-f\|_{L^p}
      +\|S_G(t)f-f\|_{L^p}\\
    &\longrightarrow0\qquad\text{as }t\downarrow0.
\end{align*}
The semigroup property gives continuity at every other time. At the
$L^\infty$ endpoint, weak-$*$ continuity follows by duality from strong
continuity on $L^1$ and the symmetry of the semigroup.
\end{proof}

In particular, the time factor is
\[
    t_\delta^{-a}
    =\begin{cases}
        t^{-a}, & 0<t\leq1,\\
        t^{-a/\delta}, & t\geq1.
    \end{cases}
\]
Thus the estimate uses the classical Grushin smoothing rate at short
times and the fractional rate at long times.

Real interpolation extends these bounds to Lorentz spaces; see
\cite{MR482275}. This is the same interpolation principle used for the
classical Grushin semigroup in \cite{MR5014751}. At the Lebesgue endpoints,
we equip $L^{1,1}=L^1$ and $L^{\infty,\infty}=L^\infty$ with their usual
norms.

\begin{theorem}
\label{Marcinkzestimate}
Let $0<\delta<1$, $1\leq p\leq r\leq\infty$, and
$1\leq s\leq\infty$. At the endpoints, assume that $s=1$ if $p=1$,
and $s=\infty$ if $r=\infty$. Then
\begin{equation}
\label{mixedlorentzestimate}
    \|e^{-t\mathcal{G}}f\|_{L^{r,s}(\mathbb{R}^{N+k})}
    \leq C\,t_\delta^{-\frac{N+2k}{2}
        \left(\frac1p-\frac1r\right)}
    \|f\|_{L^{p,s}(\mathbb{R}^{N+k})},
    \qquad t>0,
\end{equation}
where $t_\delta=\max\{t,t^{1/\delta}\}$ and $C$ is independent of $t$
and $f$.

For $1<p<\infty$ and $1\leq s<\infty$, the semigroup is strongly
continuous on $L^{p,s}(\mathbb{R}^{N+k})$. Strong continuity on
$L^{1,1}=L^1$ follows from Theorem~\ref{LpLrestimate}.
\end{theorem}

\begin{proof}
Suppose first that $1<p\leq r<\infty$. Choose exponents
$1<p_0<p<p_1<\infty$ and $1<r_0<r<r_1<\infty$ such that
\[
    \frac1{p_i}-\frac1{r_i}
    =\frac1p-\frac1r,
    \qquad i=0,1.
\]
There exists $\theta\in(0,1)$ for which
\[
    \frac1p=\frac{1-\theta}{p_0}+\frac{\theta}{p_1},
    \qquad
    \frac1r=\frac{1-\theta}{r_0}+\frac{\theta}{r_1}.
\]
Apply Theorem~\ref{LpLrestimate} to the pairs $(p_0,r_0)$ and
$(p_1,r_1)$. Both bounds have the same time factor. The real-interpolation
identities
\[
    (L^{p_0},L^{p_1})_{\theta,s}=L^{p,s},
    \qquad
    (L^{r_0},L^{r_1})_{\theta,s}=L^{r,s},
\]
with equivalence of norms, give \eqref{mixedlorentzestimate}.

For $p=1<r<\infty$ and $s=1$, interpolate the target spaces in the
Lebesgue bounds $L^1\to L^{r_0}$ and $L^1\to L^{r_1}$, where
$1<r_0<r<r_1<\infty$. This yields the required $L^1\to L^{r,1}$
estimate with the stated time factor. By symmetry and Lorentz duality,
its dual estimate gives the case $1<p<r=\infty$, $s=\infty$.
The cases $p=r=1$ and $p=r=\infty$ are already covered by
Theorem~\ref{LpLrestimate}.

Finally, for $1<p<\infty$ and $s<\infty$, smooth compactly supported
functions are dense in $L^{p,s}$. Strong continuity on two neighboring
Lebesgue spaces and interpolation give convergence at $t=0$ for such
functions. The uniform $L^{p,s}$ bound in \eqref{mixedlorentzestimate},
with $p=r$, then extends this convergence to every $f\in L^{p,s}$.
\end{proof}

For weak-$L^p$ data, the appropriate initial trace is expressed through
the predual identification
\[
    (L^{p',1})^*=L^{p,\infty},
    \qquad 1<p<\infty,
\]
with equivalent norms. We use the integral pairing
\[
    \langle f,g\rangle
    :=\int_{\mathbb{R}^{N+k}}f(x,y)g(x,y)\,\mathrm{d}x\,\mathrm{d}y.
\]
For $\varphi\in L^{p,\infty}$ and $\psi\in L^{p',1}$, positivity,
Lorentz H\"older's inequality, and \eqref{mixedlorentzestimate} imply
\[
    \int_{\mathbb{R}^{N+k}}
    \bigl(e^{-t\mathcal{G}}|\varphi|\bigr)|\psi|
    \,\mathrm{d}x\,\mathrm{d}y
    \leq C\|\varphi\|_{L^{p,\infty}}\|\psi\|_{L^{p',1}}<\infty.
\]
The mixed semigroup has a nonnegative symmetric kernel, as follows
from its subordination representation. The preceding bound justifies
Fubini's theorem and therefore gives
\begin{equation}
\label{langle}
    \left\langle e^{-t\mathcal{G}}\varphi,\psi\right\rangle
    =\left\langle\varphi,e^{-t\mathcal{G}}\psi\right\rangle,
    \qquad t>0.
\end{equation}
Since $e^{-t\mathcal{G}}$ is strongly continuous on $L^{p',1}$,
\[
    \left\langle e^{-t\mathcal{G}}\varphi-\varphi,\psi\right\rangle
    =\left\langle\varphi,e^{-t\mathcal{G}}\psi-\psi\right\rangle
    \longrightarrow0\qquad\text{as }t\downarrow0.
\]
Consequently,
\[
    e^{-t\mathcal{G}}\varphi
    \overset{*}{\rightharpoonup}\varphi
    \quad\text{in }L^{p,\infty}(\mathbb{R}^{N+k})
    \quad\text{as }t\downarrow0.
\]
Strong continuity at $t=0$ does not hold on the full space
$L^{p,\infty}$.

\section{Local well-posedness and a blow-up alternative}

In this section, we establish local existence, uniqueness, and
continuous dependence on the initial data for the nonlinear
evolution equation driven by the mixed Grushin operator. We
also derive a continuation criterion: a solution with a finite
maximal existence time must become unbounded in the underlying
Marcinkiewicz norm. The main analytical step is to combine the
integrability gain of the Riesz potential with the short-time
smoothing estimates established in the preceding section.
The Marcinkiewicz-space framework accommodates singular initial
profiles that need not belong to the corresponding Lebesgue
space, as in the semilinear Grushin setting considered
in~\cite{MR5014751}.

As before, write $d=N+k$, $Q=N+2k$, and
$z=(x,y)\in\mathbb{R}^{N}\times\mathbb{R}^{k}$.
We consider the Cauchy problem
\begin{equation}
\label{Grushineq}
\begin{cases}
\partial_t u(t,z)+\mathcal{G}u(t,z)
    =I_{\alpha}(|u(t,\cdot)|^{\rho})(z),
    & t>0,\quad z\in\mathbb{R}^{d},\\[1mm]
u(0,z)=u_0(z),
    & z\in\mathbb{R}^{d},
\end{cases}
\end{equation}
where
\[
\mathcal{G}=G+G^{\delta},
\qquad 0<\delta<1,\qquad 0<\alpha<d,\qquad \rho>1.
\]
Here, $I_{\alpha}$ is the potential operator defined
in~\eqref{Rieszpotential}; in convolution notation,
\[
I_{\alpha}f
=A_{\alpha}\bigl(|\cdot|^{-\alpha}*f\bigr).
\]
The fixed positive normalization factor $A_{\alpha}$ will be
absorbed into the constants in our estimates. With this
convention, $\alpha$ denotes the decay exponent of the kernel,
and the corresponding Riesz potential has order $d-\alpha$.

For $u_0\in L^{p,\infty}(\mathbb{R}^{d})$, with
$1<\rho<p<\infty$, we seek mild solutions satisfying
\begin{equation}
\label{local:mild}
u(t)=e^{-t\mathcal{G}}u_0
    +\int_0^t e^{-(t-s)\mathcal{G}}
        I_{\alpha}(|u(s)|^{\rho})\,\mathrm{d}s,
\qquad 0<t<T,
\end{equation}
with
\[
\sup_{0<t<T}\|u(t)\|_{L^{p,\infty}(\mathbb{R}^{d})}<\infty.
\]
The time integral is understood by duality with
$L^{p',1}(\mathbb{R}^{d})$, and the initial condition is
interpreted in the weak-$*$ sense:
\[
\lim_{t\downarrow0}\langle u(t),\psi\rangle
=\langle u_0,\psi\rangle,
\qquad
\psi\in L^{p',1}(\mathbb{R}^{d}),
\qquad \frac1p+\frac1{p'}=1.
\]

To express the condition underlying the local theory, define
$\beta>0$ by
\begin{equation}
\label{local:beta}
\frac1\beta
=\frac2Q+1-\frac\alpha d
=\frac{2}{N+2k}+1-\frac{\alpha}{N+k}.
\end{equation}
The role of this parameter follows from the nonlinear
estimates. Indeed, if $r$ is determined by
\[
\frac1r=\frac\rho p+\frac\alpha d-1,
\qquad 1<r\leq p,
\]
then the potential and semigroup bounds lead to the time
factor $(t-s)^{-\theta}$ in the Duhamel estimate, where
\begin{equation}
\label{local:theta}
\theta
=\frac Q2\left(\frac1r-\frac1p\right)
=\frac Q2\left(\frac{\rho-1}{p}+\frac\alpha d-1\right).
\end{equation}
Since $\theta\geq0$, this factor is integrable near $s=t$
precisely when $\theta<1$, or equivalently,
\[
\rho<1+\frac p\beta.
\]
Under the admissibility conditions above, this inequality
provides the small-time factor required for the contraction
argument. We now state the local well-posedness result.
\begin{theorem}[Local well-posedness]
\label{Localtheorem}
Let $d=N+k$, $0<\delta<1$, $0<\alpha<d$, and
$1<\rho<p<\infty$. Define $r$ by
\[
\frac1r=\frac\rho p+\frac\alpha d-1,
\]
and assume that $1<r\leq p$ and
\begin{equation}
\label{firstcondition}
\rho<1+\frac p\beta,
\end{equation}
where $\beta$ is defined in~\eqref{local:beta}.
For $T>0$, set
\[
X(T):=L^{\infty}\bigl((0,T);L^{p,\infty}(\mathbb{R}^{d})\bigr),
\qquad
\|u\|_{X(T)}
:=\operatorname*{ess\,sup}_{0<t<T}
\|u(t)\|_{L^{p,\infty}(\mathbb{R}^{d})}.
\]

For every $R>0$, there exists $T_R>0$ such that, for each
$u_0\in L^{p,\infty}(\mathbb{R}^{d})$ satisfying
\[
\|u_0\|_{L^{p,\infty}(\mathbb{R}^{d})}\leq R,
\]
problem~\eqref{Grushineq} admits a unique mild solution
$u\in X(T_R)$ satisfying~\eqref{local:mild}.
The solution attains its initial datum in the weak-$*$ sense:
\[
\lim_{t\downarrow0}\langle u(t)-u_0,\psi\rangle=0,
\qquad
\psi\in L^{p',1}(\mathbb{R}^{d}),
\qquad p'=\frac{p}{p-1}.
\]

Moreover, if $u_0,v_0\in L^{p,\infty}(\mathbb{R}^{d})$
satisfy
\[
\max\bigl\{
\|u_0\|_{L^{p,\infty}(\mathbb{R}^{d})},
\|v_0\|_{L^{p,\infty}(\mathbb{R}^{d})}
\bigr\}\leq R,
\]
then their corresponding solutions $u,v\in X(T_R)$ obey
\begin{equation}
\label{local:continuousdependence}
\|u-v\|_{X(T_R)}
\leq C_R\|u_0-v_0\|_{L^{p,\infty}(\mathbb{R}^{d})}.
\end{equation}
Here, $T_R$ and $C_R$ depend only on $R$ and the fixed
parameters $N,k,\delta,\alpha,\rho,p$.

\end{theorem}
\begin{proof}
Write $S(t)=e^{-t\mathcal{G}}$ and
$\mathcal{N}(w)=I_{\alpha}(|w|^{\rho})$. Throughout the proof,
all Lorentz spaces are defined over $\mathbb{R}^{d}$. Set
\[
E=L^{p,\infty}(\mathbb{R}^{d}),
\qquad F=L^{r,\infty}(\mathbb{R}^{d}),
\]
and equip these spaces with the equivalent Banach norms
introduced in the preliminaries. By $1<r\leq p$ and
\eqref{firstcondition}, the exponent
\[
\theta
=\frac Q2\left(\frac1r-\frac1p\right)
=\frac Q2\left(\frac{\rho-1}{p}+\frac\alpha d-1\right)
\]
satisfies $0\leq\theta<1$. Theorem~\ref{Marcinkzestimate}
therefore gives constants $C_0\geq1$ and $C_S>0$ such that
\begin{equation}
\label{local:linear-bounds}
\|S(t)f\|_E\leq C_0\|f\|_E,
\qquad
\|S(t)g\|_E\leq C_S t^{-\theta}\|g\|_F,
\qquad 0<t\leq1.
\end{equation}
All constants below depend only on the fixed parameters
unless otherwise indicated.

Since $0<\alpha<d$, the kernel $|\cdot|^{-\alpha}$ belongs
to $L^{d/\alpha,\infty}$. Moreover,
\[
1+\frac1r=\frac\alpha d+\frac\rho p,
\qquad \frac d\alpha>1,
\qquad \frac p\rho>1.
\]
Theorem~\ref{hardysobolevinequlity}, together with
Lemma~\ref{fractionallemma}, consequently yields
\begin{align}
\label{local:nonlinear-lipschitz}
\|\mathcal{N}(f)-\mathcal{N}(g)\|_F
&\leq C\bigl\||f|^\rho-|g|^\rho\bigr\|_{L^{p/\rho,\infty}}
\nonumber\\
&\leq C_N
\bigl(\|f\|_E^{\rho-1}+\|g\|_E^{\rho-1}\bigr)
\|f-g\|_E,
\qquad f,g\in E.
\end{align}
Here the normalization factor $A_\alpha$ and the Lorentz norm
of the potential kernel are absorbed into the constants.
Taking $g=0$ also gives
\[
\|\mathcal{N}(f)\|_F\leq C_N\|f\|_E^\rho.
\]

For $0<T\leq1$ and $u\in X(T)$, define the Duhamel map by
\begin{equation}
\label{Duhamelgru}
\mathcal{J}_{u_0}(u)(t)
=S(t)u_0+\int_0^t S(t-s)\mathcal{N}(u(s))\,\mathrm{d}s
=:S(t)u_0+\mathcal{D}(u)(t).
\end{equation}
The integral is well defined in $E$. Indeed, spectral calculus
gives operator-norm continuity of $S(t)$ on $L^2$ for $t>0$.
Interpolation with its $L^1$ and $L^\infty$ bounds, followed
by real interpolation, gives operator-norm continuity on
$L^{a,\infty}$ for every $1<a<\infty$ at positive times.
The semigroup property and~\eqref{local:linear-bounds}
then give operator-norm continuity from $F$ to $E$ away
from $t=0$. Together with~\eqref{local:nonlinear-lipschitz},
this ensures strong measurability of the integrand.
Its Bochner integrability follows from
\begin{align}
\label{eqn23}
\|\mathcal{D}(u)(t)\|_E
&\leq C_S\int_0^t
    (t-s)^{-\theta}\|\mathcal{N}(u(s))\|_F\,\mathrm{d}s
\nonumber\\
&\leq C_SC_N\|u\|_{X(T)}^\rho
    \int_0^t(t-s)^{-\theta}\,\mathrm{d}s
\nonumber\\
&=C_\theta t^{1-\theta}\|u\|_{X(T)}^\rho,
\qquad C_\theta:=\frac{C_SC_N}{1-\theta}.
\end{align}
These Bochner integrals agree with the duality interpretation
in~\eqref{local:mild}. Splitting the integral into a part away
from $s=t$ and a short terminal interval shows that
$\mathcal{D}(u)$ is norm-continuous in time; the latter part
is controlled by~\eqref{eqn23}. In particular,
$\mathcal{D}(u)\in C([0,T);E)$ after setting
$\mathcal{D}(u)(0)=0$, and $\mathcal{J}_{u_0}(u)\in X(T)$.

Taking the essential supremum in~\eqref{eqn23}, we obtain
\begin{equation}
\label{eqn123}
\|\mathcal{J}_{u_0}(u)\|_{X(T)}
\leq C_0\|u_0\|_E
    +C_\theta T^{1-\theta}\|u\|_{X(T)}^\rho.
\end{equation}
Similarly, \eqref{local:nonlinear-lipschitz} gives
\begin{equation}
\label{local:duhamel-lipschitz}
\begin{split}
\|\mathcal{D}(u)-\mathcal{D}(v)\|_{X(T)}
\leq{}&C_\theta T^{1-\theta}
\bigl(\|u\|_{X(T)}^{\rho-1}
    +\|v\|_{X(T)}^{\rho-1}\bigr)\\
&\times\|u-v\|_{X(T)}.
\end{split}
\end{equation}

Fix $R>0$ and suppose that $\|u_0\|_E\leq R$.
Set $M_R=2C_0R$ and choose $T_R\in(0,1]$ such that
\begin{equation}
\label{local:existence-time}
C_\theta T_R^{1-\theta}M_R^{\rho-1}\leq\frac14.
\end{equation}
This choice depends only on $R$ and the fixed parameters.
Consider the closed ball
\[
\mathbb{B}_R
:=\{u\in X(T_R):\|u\|_{X(T_R)}\leq M_R\}.
\]
For $u\in\mathbb{B}_R$, estimates~\eqref{eqn123} and
\eqref{local:existence-time} imply
\[
\|\mathcal{J}_{u_0}(u)\|_{X(T_R)}
\leq\frac{M_R}{2}+\frac{M_R}{4}
\leq M_R.
\]
Moreover, for $u,v\in\mathbb{B}_R$,
\eqref{local:duhamel-lipschitz} yields
\[
\|\mathcal{J}_{u_0}(u)-\mathcal{J}_{u_0}(v)\|_{X(T_R)}
\leq2C_\theta T_R^{1-\theta}M_R^{\rho-1}
    \|u-v\|_{X(T_R)}
\leq\frac12\|u-v\|_{X(T_R)}.
\]
Thus $\mathcal{J}_{u_0}$ is a contraction of
$\mathbb{B}_R$ into itself. Since $X(T_R)$ is a Banach
space, the Banach fixed-point theorem gives a unique fixed
point in $\mathbb{B}_R$. Choosing the representative defined
by~\eqref{Duhamelgru}, we obtain a mild solution for every
$0<t<T_R$.

To verify the initial trace, first note that
\[
\|u(t)-S(t)u_0\|_E
\leq C_\theta t^{1-\theta}M_R^\rho
\longrightarrow0
\qquad\text{as }t\downarrow0.
\]
For every $\psi\in L^{p',1}$, the pairing identity
\eqref{langle} gives
\[
\langle u(t)-u_0,\psi\rangle
=\langle u_0,S(t)\psi-\psi\rangle
    +\langle\mathcal{D}(u)(t),\psi\rangle.
\]
The first term tends to zero by strong continuity on
$L^{p',1}$, while the second tends to zero by the preceding
estimate and Lorentz H\"older's inequality. Hence
$u(t)\overset{*}{\rightharpoonup} u_0$ in $E$ as $t\downarrow0$.

We next establish uniqueness among all mild solutions in
$X(T_R)$. Let $u,v\in X(T_R)$ have the same initial datum
and set
\[
K=\|u\|_{X(T_R)}^{\rho-1}
    +\|v\|_{X(T_R)}^{\rho-1}.
\]
Choose $h\in(0,T_R]$ such that
$C_\theta h^{1-\theta}K<1$. Subtracting their mild
formulations and applying~\eqref{local:duhamel-lipschitz}
on $(0,h)$ gives
\[
\|u-v\|_{X(h)}
\leq C_\theta h^{1-\theta}K\|u-v\|_{X(h)},
\]
so $u=v$ on this interval. If equality has already been
established on $(0,a)$, subtraction of the mild formulations
for $t>a$ leaves only the integral over $(a,t)$.
The same estimate therefore proves equality on
$(a,\min\{a+h,T_R\})$. Iterating over finitely many such
intervals gives $u=v$ throughout $(0,T_R)$.

Finally, let $\|u_0\|_E,\|v_0\|_E\leq R$. Their
corresponding solutions lie in $\mathbb{B}_R$, and hence
\begin{align*}
\|u-v\|_{X(T_R)}
&\leq C_0\|u_0-v_0\|_E
    +\|\mathcal{D}(u)-\mathcal{D}(v)\|_{X(T_R)}\\
&\leq C_0\|u_0-v_0\|_E
    +\frac12\|u-v\|_{X(T_R)}.
\end{align*}
Absorbing the last term yields
\[
\|u-v\|_{X(T_R)}
\leq2C_0\|u_0-v_0\|_E,
\]
which proves~\eqref{local:continuousdependence} and
completes the proof.
\end{proof}

We next discuss the continuation of local solutions.
Since~\eqref{Grushineq} is autonomous, the local theory can
be restarted at any positive time $t_0$ within the interval
of existence, taking $u(t_0)$ as the new initial datum.
By Theorem~\ref{Localtheorem}, the resulting lifespan can
be chosen uniformly when these data range over a bounded
subset of $L^{p,\infty}(\mathbb{R}^{d})$.

Let
\[
u:[0,T)\longrightarrow L^{p,\infty}(\mathbb{R}^{d})
\]
be a mild solution of~\eqref{Grushineq}. Here and below,
a mild solution on $[0,T)$ is required to belong to $X(\tau)$
for every $0<\tau<T$, and the initial value $u(0)=u_0$
is understood in the weak-$*$ sense specified above.
A mild solution
\[
\overline{u}:[0,\overline{T})
\longrightarrow L^{p,\infty}(\mathbb{R}^{d}),
\qquad \overline{T}>T,
\]
is called a \emph{continuation} of $u$ if
\[
\overline{u}(t)=u(t),
\qquad 0\leq t<T.
\]

A mild solution defined on $[0,T_m)$, where
$T_m\in(0,\infty]$, is called \emph{maximal} if it admits
no continuation to a strictly larger time interval.
The number $T_m$ is its \emph{maximal existence time};
when $T_m=\infty$, the solution is called \emph{global}.
The following theorem establishes the existence of a unique
maximal mild solution and gives the corresponding blow-up
alternative in the Marcinkiewicz norm.
\begin{theorem}[Maximal existence and blow-up alternative]
\label{Maximaltheorem}
Let $d=N+k$, $0<\delta<1$, $0<\alpha<d$, and
$1<\rho<p<\infty$. Suppose that
\[
\frac1r=\frac\rho p+\frac\alpha d-1,
\qquad 1<r\leq p,
\]
and
\begin{equation}
\label{hypo1}
\rho<1+\frac p\beta,
\end{equation}
where $\beta$ is defined in~\eqref{local:beta}.
For every $u_0\in L^{p,\infty}(\mathbb{R}^{d})$,
problem~\eqref{Grushineq} admits a unique maximal mild
solution
\[
u:[0,T_m)\longrightarrow L^{p,\infty}(\mathbb{R}^{d}),
\qquad T_m\in(0,\infty],
\]
with $u\in X(T)$ for every $0<T<T_m$.
If $T_m<\infty$, then
\begin{equation}
\label{ftb}
\lim_{t\uparrow T_m}
\|u(t)\|_{L^{p,\infty}(\mathbb{R}^{d})}=+\infty.
\end{equation}
\end{theorem}

\begin{proof}
Write $E=L^{p,\infty}(\mathbb{R}^{d})$,
$S(t)=e^{-t\mathcal{G}}$, and
$\mathcal{N}(w)=I_\alpha(|w|^\rho)$.
Theorem~\ref{Localtheorem} provides a local mild solution
for the initial datum $u_0$. Define
\[
\mathcal{T}
:=\{T>0:\text{there is a mild solution on }[0,T)
\text{ with initial datum }u_0\},
\qquad
T_m:=\sup\mathcal{T}.
\]
Here mild solutions are understood in the locally bounded
class specified above. In particular,
$\mathcal{T}\neq\varnothing$ and $T_m\in(0,\infty]$.
The uniqueness argument in the proof of
Theorem~\ref{Localtheorem}, applied on successive short
intervals, shows that any two such solutions coincide
throughout their common interval of existence.
They therefore define a single mild solution on $[0,T_m)$.
For each $T<T_m$, one may choose $\widehat T\in\mathcal{T}$
with $T<\widehat T$; hence this solution belongs to $X(T)$.
By the definition of $T_m$, it admits no continuation to a
larger interval. This proves existence and uniqueness of
the maximal mild solution.

The semigroup property allows the mild formulation to be
restarted at every $t_0\in(0,T_m)$:
\begin{equation}
\label{local:restart}
u(t)=S(t-t_0)u(t_0)
    +\int_{t_0}^{t}S(t-s)\mathcal{N}(u(s))\,\mathrm{d}s,
\qquad t_0<t<T_m.
\end{equation}
Indeed, this identity follows by splitting the Duhamel
integral at $t_0$ and using the mild formulation for
$u(t_0)$.

Suppose now that $T_m<\infty$ and that~\eqref{ftb} fails.
Then there exist $R>0$ and a sequence
$t_n\uparrow T_m$, with $0<t_n<T_m$, such that
\[
\|u(t_n)\|_E\leq R
\qquad\text{for every }n.
\]
By Theorem~\ref{Localtheorem}, there is a time
$\tau_R>0$, independent of $n$, such that the problem with
initial datum $u(t_n)$ admits a mild solution
$v_n\in X(\tau_R)$. Thus
\[
v_n(s)=S(s)u(t_n)
    +\int_0^s S(s-\sigma)\mathcal{N}(v_n(\sigma))
        \,\mathrm{d}\sigma,
\qquad 0<s<\tau_R.
\]
By~\eqref{local:restart} and uniqueness on every compact
subinterval of the common existence interval,
\[
v_n(s)=u(t_n+s),
\qquad 0\leq s<\min\{\tau_R,T_m-t_n\}.
\]

Choose $n$ sufficiently large that $t_n+\tau_R>T_m$,
and define
\[
\widetilde u(t)
:=\begin{cases}
u(t), & 0\leq t\leq t_n,\\[1mm]
v_n(t-t_n), & t_n<t<t_n+\tau_R.
\end{cases}
\]
This function agrees with $u$ on $[0,T_m)$ and is locally
bounded in $E$. Substituting the mild formulation for
$u(t_n)$ into that for $v_n$, and using the semigroup
property, gives
\[
\widetilde u(t)=S(t)u_0
    +\int_0^t S(t-s)\mathcal{N}(\widetilde u(s))
        \,\mathrm{d}s,
\qquad 0<t<t_n+\tau_R.
\]
Consequently, $\widetilde u$ is a mild continuation of $u$
beyond $T_m$, contradicting maximality. This proves
\eqref{ftb}.
\end{proof}

\section{Global well-posedness}

We now consider the endpoint
\[
\rho=1+\frac p\beta,
\qquad
\frac1\beta=\frac2Q+1-\frac\alpha d,
\qquad d=N+k,\quad Q=N+2k.
\]
At this endpoint, the exponent in~\eqref{local:theta} equals
one, and the pointwise semigroup estimate produces the
nonintegrable time kernel $(t-s)^{-1}$. To control the
nonlinear term, we combine Lorentz duality with a
time-integrated semigroup estimate. This yields global mild
solutions for sufficiently small initial data, together with
uniqueness and Lipschitz dependence in a fixed small ball of
the solution space.

We first record the mixed Grushin version of the Yamazaki
estimate; see~\cite{MR5014751,MR1777114}.
Throughout this section, $S(t)=e^{-t\mathcal{G}}$.

\begin{lemma}[Yamazaki-type estimate]
\label{Yamazaki}
Let $0<\delta<1$ and $1<a<b<\infty$. Then
\[
\int_0^\infty
t^{\frac Q2(\frac1a-\frac1b)-1}
\|S(t)\phi\|_{L^{b,1}(\mathbb{R}^{d})}\,\mathrm{d}t
\leq C\|\phi\|_{L^{a,1}(\mathbb{R}^{d})},
\qquad \phi\in L^{a,1}(\mathbb{R}^{d}),
\]
where $C$ depends only on $N,k,\delta,a,b$.
\end{lemma}

\begin{proof}
The subordinated semigroup $e^{-tG^\delta}$ is contractive
on the Lebesgue spaces. Real interpolation therefore gives
\[
\|e^{-tG^\delta}f\|_{L^{b,1}}
\leq C_b\|f\|_{L^{b,1}},
\qquad t>0.
\]
Using the commuting factorization
$S(t)=e^{-tG^\delta}S_G(t)$, we obtain
\[
\|S(t)\phi\|_{L^{b,1}}
\leq C_b\|S_G(t)\phi\|_{L^{b,1}}.
\]
Multiplication by $t^{\frac Q2(\frac1a-\frac1b)-1}$
and integration over $(0,\infty)$ reduce the conclusion to
the Yamazaki estimate for the classical Grushin semigroup
established in~\cite{MR5014751}.
\end{proof}

\begin{theorem}[Global well-posedness for small initial data]
\label{Globaltheorem}
Let $0<\delta<1$, $0<\alpha<d$, and $1<\rho<p<\infty$
satisfy
\begin{equation}
\label{hypo2}
\rho=1+\frac p\beta.
\end{equation}
Set
\[
E=L^{p,\infty}(\mathbb{R}^{d}),
\qquad X=L^\infty\bigl((0,\infty);E\bigr).
\]
There exist constants $\eta,\varepsilon >0$ such that,
whenever $u_0\in E$ and $\|u_0\|_E\leq\eta$,
problem~\eqref{Grushineq} admits a global mild solution
\[
u\in X
\qquad \|u\|_X\leq\varepsilon.
\]
Moreover, the initial datum is attained in the
weak-$*$ sense:
\[
u(t)\stackrel{*}\rightharpoonup u_0
\quad\text{in }E\quad\text{as }t\downarrow0.
\]
The solution is unique among global mild solutions in
\[
B_{\varepsilon}(X)
:=\{v\in X:\|v\|_X\leq\varepsilon\}.
\]
Moreover, solutions corresponding to initial data
$u_0,v_0\in E$ with
$\|u_0\|_E,\|v_0\|_E\leq\eta$ satisfy
\begin{equation}
\label{global:continuousdependence}
\|u-v\|_X\leq C\|u_0-v_0\|_E.
\end{equation}
In particular, $\|u\|_X\leq C\|u_0\|_E$.
All constants appeared here depend only on the fixed parameters
$N,k,\delta,\alpha,\rho,p$.
\end{theorem}

\begin{proof}
Let $\mathcal{J}$ be the function defined in Theorem~\ref{Localtheorem}. By Theorem~\ref{Marcinkzestimate}, the linear part satisfies
\begin{align*}
\left\|S(t)u_{0}\right\|_{E}
\leq
C_{0}'\left\|u_{0}\right\|_{E},
\end{align*}
for some constant $C_{0}'>0$.

We estimate the inhomogeneous part as follows. We denote the nonlinear term as 
\begin{align*}
\mathcal{N}(u)=I_{\alpha}(|u|^{\rho}).
\end{align*}Using Theorem \ref{generalisedholder} and employing \eqref{langle} we have 
\begin{align}
\label{bigeqn}
\nonumber\bigg\|&\int_{0}^{t} S(t-s)\mathcal{N}(u(s))\ \mathrm{d}s\bigg\|_{L^{(p,\infty)}}\\
\nonumber&=\sup\limits_{\|\varphi\|_{L^{(p',1)}}=1}\bigg|\int_{\mathbb{R}^{d}}\left(\int_{0}^{t}S(t-s)\mathcal{N}(u(s))\ \mathrm{d}s\right)\varphi(y)\ \mathrm{d}y\bigg|\\
\nonumber&=\sup\limits_{\|\varphi\|_{L^{(p',1)}}=1}\bigg|\int_{0}^{t}\int_{\mathbb{R}^{d}}S(t-s)\mathcal{N}(u(s))\varphi(y)\ \mathrm{d}y\ \mathrm{d}s\bigg|\\
\nonumber&=\sup\limits_{\|\varphi\|_{L^{(p',1)}}=1}\bigg|\int_{0}^{t}\int_{\mathbb{R}^{d}}\mathcal{N}(u(s))S(t-s)\varphi(y)\ \mathrm{d}y\ \mathrm{d}s\bigg|\\
\nonumber&\leq\sup\limits_{\|\varphi\|_{L^{(p',1)}}=1}\int_{0}^{t}\int_{\mathbb{R}^{d}}\mathcal{N}(u(s))S(t-s)\varphi(y)\big|\ \mathrm{d}y\ \mathrm{d}s\\
&\leq\sup\limits_{\|\varphi\|_{L^{(p',1)}}=1}\int_{0}^{t}\|\mathcal{N}(u(s))\|_{L^{r,\infty}}\|S(t-s)\varphi(y)\|_{L^{r,1}}\ \mathrm{d}s
\end{align}

We choose $r>1$ be such that  $1\leq p\leq r\leq \infty$ and  
\begin{align*}
\frac{1}{r}=\frac{\alpha}{d}+\frac{\rho}{p}-1.
\end{align*}
Then by Theorem \ref{hardysobolevinequlity} we have 
\begin{align*}
\|I_{\alpha}(|u(s)|^{\rho})\|_{L^{r,\infty}}&\leq C\||x|^{-\alpha}\|_{L^{\frac{d}{\alpha}, \infty}}\||u(s)|^{\rho}\|_{L^{\frac{p}{\rho},\infty}}\\
&\leq C'\|u(s)\|_{L^{p,\infty}}^{\rho}\\
&\leq C'\sup\limits_{s\in (0,\infty)}\|u(s)\|_{L^{p,\infty}}^{\rho}=\|u\|_{X}^{\rho}.
\end{align*}

Employing this estimate from \eqref{bigeqn} we have 
\begin{align}
\label{lteqn}
\bigg\|&\int_{0}^{t} S(t-s)I_{\alpha}(|u(s)|^{\rho})\ \mathrm{d}s\bigg\|_{E} \leq C'\|u\|^{\rho}_{X}\sup\limits_{\|\varphi\|_{L^{p',1}}=1}\int_{0}^{t}\|S(t-s)\varphi(y)\|_{L^{r,1}}\ \mathrm{d}s
\end{align}
Note that by the  hypothesis \eqref{hypo2} we have 
\begin{align*}
\frac{\rho-1}{p}=\frac{2}{d}-\frac{\alpha}{d}+1\Rightarrow \frac{Q}{2}\left(\frac{\alpha}{d}-\frac{\rho-1}{p}-1\right)=1.
\end{align*}
Therefore
\begin{align*}
&\frac{Q}{2}\left(\frac{1}{p'}-\frac{1}{r'}\right)-1\\
&=\frac{Q}{2}\left(\frac{1}{r}-\frac{1}{p}\right)-1\\
&=\frac{Q}{2}\left(\frac{\alpha}{d}+\frac{\rho-1}{p}-1\right)-1=0.
\end{align*}
Thus by using Yamazaki-type  inequality, Theorem \ref{Yamazaki} we have 
\begin{align*}
&\sup\limits_{\|\varphi\|_{L^{p',1}}=1}\int_{0}^{t}\|S(t-s)\varphi\|_{L^{r',1}}\ \mathrm{d}s\\
&=\sup\limits_{\|\varphi\|_{L^{p',1}}=1}\int_{0}^{t}\|S(s)\varphi\|_{L^{r',1}}\ \mathrm{d}s\\
&\leq \sup\limits_{\|\varphi\|_{L^{p',1}}=1}\int_{0}^{\infty}\|S(s)\varphi\|_{L^{r',1}}\ \mathrm{d}s\\
&=\sup\limits_{\|\varphi\|_{L^{p',1}}=1}\int_{0}^{\infty}s^{\frac{Q}{2}\left(\frac{1}{p'}-\frac{1}{r'}\right)-1}\|S(s)\varphi\|_{L^{r',1}}\ \mathrm{d}s\\
&\leq C_{1}'\sup\limits_{\|\varphi\|_{L^{p',1}}=1}\|\varphi\|_{L^{p', 1}}= C_{1}',
\end{align*}
for some constant $C_{1}'>0$. Thus from \eqref{lteqn} we finally obtain 
\begin{align*}
\bigg\|&\int_{0}^{t} S(s)I_{\alpha}(|u(s)|^{\rho})\ \mathrm{d}s\bigg\|_{L^{p,\infty}}\leq C'' \|u\|_{X}^{\rho}.
\end{align*}
From which we can conclude
\begin{align}
\label{ballestimaet}
\|\mathcal{J}(u)\|_{X}\leq C_{0}(\|u_{0}\|_{E}+\|u\|_{X}^{\rho})
\end{align}
 with some uniform constant $C_{0}$.
 For $\varepsilon>0$, define
\[
B_{\varepsilon}
=
\Bigl\{
u\in X
:\ 
\|u\|_{X}\leq \varepsilon
\Bigr\},
\]
which is the closed ball of radius $\varepsilon$ centered at the origin in $X$.
 We next show that \(\mathcal{J}\) maps \(B_{\varepsilon}\) into itself for a suitable choice of \(\varepsilon>0\). Suppose that the initial datum satisfies
\[
\|u_{0}\|_{E}
\leq
\frac{\varepsilon}{2C_{0}}.
\]
Then, for any \(u\in B_{\varepsilon}\), estimate \eqref{ballestimaet} yields
\[
\|\mathcal{J}(u)\|_{X}
\leq
\frac{\varepsilon}{2}
+
C_{0}\varepsilon^{3}.
\]

Choosing \(\varepsilon>0\) sufficiently small so that
\[
C_{0}\varepsilon^{\rho-1}\leq \frac12,
\]
we obtain
\[
C_{0}\varepsilon^{\rho-1}
\leq
\frac{\varepsilon}{2}.
\]
Therefore,
\[
\|\mathcal{J}(u)\|_{X}
\leq
\frac{\varepsilon}{2}
+
\frac{\varepsilon}{2}
=
\varepsilon.
\]

It follows that \(\mathcal{J}(u)\in B_{\varepsilon}\) whenever \(u\in B_{\varepsilon}\). Hence \(\mathcal{J}\) maps \(B_{\varepsilon}\) into itself.

Using Lemma~\ref{fractionallemma} and arguing similarly as above, we can show that
\begin{align}
\label{global:nonlinear-bounds}
\sup\limits_{t>0}\|\int_{0}^{t}S(t-s)(\mathcal{N}(u(s))-\mathcal{N}(v(s)))\ \mathrm{d}s\|_{E}\leq C\left(\|u\|^{\rho-1}_{X}+\|v\|_{X}^{\rho-1}\right)\|u-v\|_{X}
\end{align}

for sufficiently small $\varepsilon>0$, the operator $\mathcal{J}$ is a contraction on the ball

$$
B_{\varepsilon}\subset X.
$$

Therefore, by Banach's fixed-point theorem, $\mathcal{J}$ admits a unique fixed point in $B(0,\varepsilon)$. This fixed point is the desired global-in-time mild solution of \eqref{Grushineq}.

For each fixed $\phi\in L^{p',1}$ by~\eqref{langle}
and strong continuity on $L^{p',1}$,
\[
\langle S(t)u_0-u_0,\phi\rangle
=\langle u_0,S(t)\phi-\phi\rangle\longrightarrow0.
\]
Hence $u(t)\overset{*}{\rightharpoonup} u_0$ in $E$.

Finally, let $u,v$ correspond to two initial data satisfying
the smallness condition. Their mild formulations and
\eqref{global:nonlinear-bounds} imply
\[
\|u-v\|_X
\leq C_0\|u_0-v_0\|_E+\frac12\|u-v\|_X.
\]
Therefore
\[
\|u-v\|_X\leq2C_0\|u_0-v_0\|_E.
\]
Taking $v_0=0$, whose solution is $v=0$, also gives
$\|u\|_X\leq2C_0\|u_0\|_E$. This completes the proof.
\end{proof}
\subsection{The supercritical regime}

We now consider the regime which is supercritical with respect to the
Marcinkiewicz space \(L^{p,\infty}(\mathbb R^d)\), namely
\begin{equation}
\rho>1+\frac{p}{\beta},
\qquad
\frac1\beta=\frac2Q+1-\frac{\alpha}{d}.
\end{equation}
In this range, the exponent

$$
\theta_p
=
\frac Q2
\left(
\frac{\rho-1}{p}+\frac{\alpha}{d}-1
\right)
$$

satisfies \(\theta_p>1\). Consequently, the estimate used in
Section~4 produces the nonintegrable kernel

$$
(t-s)^{-\theta_p},
$$

and the contraction argument in

$$
L^\infty\bigl((0,T);L^{p,\infty}(\mathbb R^d)\bigr)
$$

cannot be extended directly to this range.

The mixed nature of the diffusion, however, allows the supercritical
nonlinearity to be treated after imposing additional spatial
integrability on the initial datum. For this purpose, define
\begin{equation}
\frac1{\beta_\delta}
:=
\frac{2\delta}{Q}+1-\frac{\alpha}{d}.
\label{eq:beta-delta}
\end{equation}
Since \(0<\delta<1\), we have

$$
\beta_\delta>\beta.
$$

For a fixed nonlinear exponent \(\rho>1\), introduce the two
integrability indices
\begin{equation}
p_c:=\beta(\rho-1),
\qquad
p_\delta:=\beta_\delta(\rho-1).
\label{eq:critical-indices-super}
\end{equation}
Thus

$$
p_c<p_\delta.
$$

Suppose now that

$$
1<\rho<p<\infty,
\qquad
\rho>1+\frac{p}{\beta}.
$$

Then \(p<p_c\). Hence the original space \(L^{p,\infty}\) lies on the
supercritical side of the short-time integrability balance. We recover
well-posedness by choosing an auxiliary exponent
\begin{equation}
p_c<q<p_\delta.
\label{eq:q-super-range}
\end{equation}
Notice that \(q>p_c>p>\rho\).

For such a \(q\), define
\begin{equation}
\vartheta
:=
\frac Q2
\left(
\frac{\rho-1}{q}
+\frac{\alpha}{d}-1
\right).
\label{eq:theta-super}
\end{equation}
The inequalities in \eqref{eq:q-super-range} are equivalent to
\begin{equation}
\delta<\vartheta<1.
\label{eq:theta-range}
\end{equation}
Indeed,

$$
q>\beta(\rho-1)
$$

gives

$$
\frac{\rho-1}{q}
<
\frac1\beta
=
\frac2Q+1-\frac{\alpha}{d},
$$

and therefore \(\vartheta<1\). Similarly,

$$
q<\beta_\delta(\rho-1)
$$

implies

$$
\frac{\rho-1}{q}
>
\frac1{\beta_\delta}
=
\frac{2\delta}{Q}
+1-\frac{\alpha}{d},
$$

which gives \(\vartheta>\delta\).

We can now establish a global small-data result.

\begin{theorem}[Supercritical well-posedness in a higher Marcinkiewicz space]
\label{thm:supercritical}
Let

$$
d=N+k,\qquad Q=N+2k,
$$

and suppose that

$$
0<\delta<1,\qquad 0<\alpha<d,
\qquad 1<\rho<p<\infty.
$$

Let \(\beta\) and \(\beta_\delta\) be defined by

$$
\frac1\beta
=
\frac2Q+1-\frac{\alpha}{d},
\qquad
\frac1{\beta_\delta}
=
\frac{2\delta}{Q}+1-\frac{\alpha}{d}.
$$

Assume that
\begin{equation}
\rho>1+\frac{p}{\beta}.
\label{eq:supercritical-p}
\end{equation}
Choose \(q\) satisfying
\begin{equation}
\beta(\rho-1)<q<
\beta_\delta(\rho-1).
\label{eq:q-choice-super}
\end{equation}

Then, for every

$$
u_0\in L^{q,\infty}(\mathbb R^d),
$$

there exists \(T>0\) such that problem \emph{(4.1)} admits a unique mild
solution

$$
u\in
L^\infty\bigl((0,T);L^{(q,\infty)}(\mathbb R^d)\bigr).
$$

Moreover, there exists \(\eta>0\) such that, whenever

$$
\|u_0\|_{L^{(q,\infty)}}\leq\eta,
$$

the corresponding solution is global and satisfies

$$
u\in
L^\infty\bigl((0,\infty);L^{q,\infty}(\mathbb R^d)\bigr)
$$

and

$$
\sup_{t>0}
\|u(t)\|_{L^{q,\infty}}
\leq
C\|u_0\|_{L^{q,\infty}}.
$$

The solution is unique in a sufficiently small ball of
\(L^\infty((0,\infty);L^{q,\infty})\), depends Lipschitz continuously
on the initial datum there, and satisfies

$$
u(t)\stackrel{*}{\rightharpoonup}u_0
\qquad\text{in }L^{q,\infty}(\mathbb R^d)
\quad\text{as }t\downarrow0.
$$

\end{theorem}

\begin{proof}
Write

$$
S(t)=e^{-t\mathcal G},
\qquad
\mathcal G=G+G^\delta,
\qquad
\mathcal N(u)=I_\alpha(|u|^\rho).
$$

Define \(r\) by
\begin{equation}
\frac1r
=
\frac{\rho}{q}+\frac{\alpha}{d}-1.
\label{eq:r-super}
\end{equation}
Since \(q>\rho\) and \(0<\alpha<d\), we have \(1/r<1\). On the other
hand, by \eqref{eq:theta-range},

$$
\frac1r-\frac1q
=
\frac{\rho-1}{q}+\frac{\alpha}{d}-1
=
\frac{2\vartheta}{Q}>0.
$$

Consequently,

$$
1<r<q<\infty.
$$

By O'Neil's convolution inequality and the fact that

$$
|\cdot|^{-\alpha}\in L^{d/\alpha,\infty}(\mathbb R^d),
$$

we obtain
\begin{equation}
\|\mathcal N(u)\|_{L^{r,\infty}}
\leq
C\|u\|_{L^{q,\infty}}^{\rho}.
\label{eq:nonlinear-super}
\end{equation}
Similarly, Lemma~2.9 and O'Neil's inequality give
\begin{align}
\|\mathcal N(u)-\mathcal N(v)\|_{L^{r,\infty}}
\leq
C
\left(
\|u\|_{L^{q,\infty}}^{\rho-1}
+
\|v\|_{L^{q,\infty}}^{\rho-1}
\right)
\|u-v\|_{L^{q,\infty}}.
\label{eq:nonlinear-diff-super}
\end{align}

Since

$$
\frac Q2
\left(
\frac1r-\frac1q
\right)
=\vartheta,
$$

Theorem~3.5 yields
\begin{equation}
\|S(t)f\|_{L^{q,\infty}}
\leq
C K_{\vartheta(t)}
\|f\|_{L^{r,\infty}},
\label{eq:semigroup-super}
\end{equation}
where
\begin{equation}
K_{\vartheta(t)}
:=
\min
\left\{
t^{-\vartheta},
t^{-\vartheta/\delta}
\right\}
=
\begin{cases}
t^{-\vartheta}, & 0<t\leq1,\\
t^{-\vartheta/\delta}, & t\geq1.
\end{cases}
\label{eq:K-super}
\end{equation}

The important point is that \(\delta<\vartheta<1\). Hence
\(K_\vartheta\in L^1(0,\infty)\), since
\begin{align}
\int_0^\infty K_\vartheta(t),dt
&=
\int_0^1t^{-\vartheta},dt
+
\int_1^\infty t^{-\vartheta/\delta},dt\\
&=
\frac1{1-\vartheta}
+
\frac{\delta}{\vartheta-\delta}
<\infty.
\label{eq:kernel-integrable}
\end{align}
Here the order-two Grushin term controls the singularity near \(t=0\),
whereas the fractional component gives the stronger decay required at
large times.

We first prove local existence. For \(0<T\leq1\), set

$$
X_q(T)
=
L^\infty\bigl((0,T);L^{q,\infty}(\mathbb R^d)\bigr)
$$

with

$$
\|u\|_{X_q(T)}
=
\operatorname*{ess\,sup}_{0<t<T}
\|u(t)\|_{L^{q,\infty}}.
$$

Define

$$
\mathcal J_{u_0}(u)(t)
=
S(t)u_0+
\int_0^t
S(t-s)\mathcal N(u(s))\,ds.
$$

Using \eqref{eq:nonlinear-super} and
\eqref{eq:semigroup-super}, we obtain
\begin{align}
\left\|
\int_0^t
S(t-s)\mathcal N(u(s)),ds
\right\|_{L^{q,\infty}}
&\leq
C\|u\|_{X_q(T)}^\rho
\int_0^t(t-s)^{-\vartheta}\ ds\\
&\leq
C T^{1-\vartheta}
\|u\|_{X_q(T)}^\rho.
\end{align}
Similarly,
\begin{align}
\|\mathcal J_{u_0}(u)-\mathcal J_{u_0}(v)\|_{X_q(T)}
\leq
C T^{1-\vartheta}
\left(
\|u\|_{X_q(T)}^{\rho-1}
+
\|v\|_{X_q(T)}^{\rho-1}
\right)
\|u-v\|_{X_q(T)}.
\end{align}
Since \(1-\vartheta>0\), the usual contraction argument gives a unique
local mild solution for arbitrary
\(u_0\in L^{q,\infty}\).

We now prove the global small-data assertion. Set

$$
X_q
=
L^\infty\bigl((0,\infty);L^{q,\infty}(\mathbb R^d)\bigr).
$$

By \eqref{eq:kernel-integrable},
\begin{align}
\left\|
\int_0^t
S(t-s)\mathcal N(u(s)),ds
\right\|_{L^{q,\infty}}
&\leq
C\|u\|_{X_q}^\rho
\int_0^tK_\vartheta(t-s)\ \mathrm{d}s\\
&\leq
C_\vartheta\|u\|_{X_q}^{\rho},
\end{align}
uniformly for \(t>0\). Hence
\begin{equation}
\|\mathcal J_{u_0}(u)\|_{X_q}
\leq
C_0\|u_0\|_{L^{q,\infty}}
+
C_\vartheta\|u\|_{X_q}^\rho.
\label{eq:global-super-map}
\end{equation}
Likewise,
\begin{align}
\|\mathcal J_{u_0}(u)-\mathcal J_{u_0}(v)\|_{X_q}
\leq
C_{\vartheta}
\left(
\|u\|_{X_q}^{\rho-1}
+
\|v\|_{X_q}^{\rho-1}
\right)
\|u-v\|_{X_q}.
\label{eq:global-super-contraction}
\end{align}

Choose \(\varepsilon_*>0\) sufficiently small that

$$
2C_\vartheta\varepsilon_*^{\rho-1}<\frac12,
$$

and then choose \(\eta_*>0\) such that

$$
C_0\eta+
C_\vartheta\varepsilon^\rho
\leq
\varepsilon.
$$

The map \(\mathcal J_{u_0}\) is then a contraction on the closed ball

$$
B_{\varepsilon}(X_q)
=
\{u\in X_q:\|u\|_{X_q}\leq\varepsilon\}.
$$

The Banach fixed-point theorem gives a global mild solution and
uniqueness in this ball.

If \(u\) and \(v\) correspond to two sufficiently small initial data,
then \eqref{eq:global-super-contraction} gives

$$
\|u-v\|_{X_q}
\leq
C\|u_0-v_0\|_{L^{q,\infty}},
$$

after absorbing the nonlinear term. Taking \(v_0=0\) yields

$$
\|u\|_{X_q}
\leq
C\|u_0\|_{L^{q,\infty}}.
$$

Finally, for \(0<t\leq1\),
\begin{align}
\left\|
\int_0^t
S(t-s)\mathcal N(u(s)),ds
\right\|_{L^{q,\infty}}
\leq
C t^{1-\vartheta}\|u\|_{X_q}^\rho
\longrightarrow0
\end{align}
as \(t\downarrow0\). The weak-\(*\) continuity of \(S(t)\) on
\(L^{q,\infty}\), established in Section~3, therefore gives

$$
u(t)\stackrel{*}{\rightharpoonup}u_0
\qquad\text{in }L^{q,\infty}
$$

as \(t\downarrow0\). This completes the proof.
\end{proof}

\begin{remark}
The condition

$$
\beta(\rho-1)<q<\beta_\delta(\rho-1)
$$

has a transparent interpretation. At the lower endpoint
\(q=\beta(\rho-1)\), one has \(\vartheta=1\), which is precisely the
short-time critical case treated in Section~5 by the Yamazaki
estimate. At the upper endpoint
\(q=\beta_\delta(\rho-1)\), one has \(\vartheta=\delta\), so that the
large-time part of the kernel behaves like \(t^{-1}\). Thus the open
interval above is exactly the range in which the pointwise mixed
semigroup estimate is integrable over the entire time axis.
\end{remark}

\begin{remark}
The condition

$$
\rho>1+\frac{p}{\beta}
$$

should be understood as supercriticality relative to the prescribed
space \(L^{p,\infty}\), and not as a sharp Fujita threshold. In fact,
for this range the short-time Duhamel kernel associated with
\(L^{p,\infty}\) has exponent strictly larger than one. Therefore the
preceding result does not assert well-posedness for arbitrary
\(L^{p,\infty}\) data. Instead, it shows that the same supercritical
nonlinearity can be treated after passing to the higher-integrability
Marcinkiewicz space \(L^{q,\infty}\).
\end{remark}


\section*{Concluding Remarks}

In this paper, we have developed a well-posedness theory in Marcinkiewicz spaces for the semilinear heat equation

$$
\partial_tu+(G+G^\delta)u=I_\alpha(|u|^\rho)
$$

associated with a mixed local--nonlocal Grushin operator, where \(0<\delta<1\), together with a spatially nonlocal Riesz-potential source. A central ingredient is the derivation of Lebesgue and Lorentz smoothing estimates for the mixed semigroup

$$
S(t)=e^{-t(G+G^\delta)},
$$

obtained through spectral calculus, subordination, and interpolation. These estimates exhibit two different diffusion scales: the second-order Grushin behaviour at short times and the fractional decay at large times. The nonlinear estimates consequently involve both the homogeneous dimension \(Q=N+2k\) of the Grushin geometry and the Euclidean dimension \(d=N+k\) associated with the potential kernel.
Considering

$$
\frac1\beta=\frac2Q+1-\frac{\alpha}{d},
$$

we obtain a well-posedness theory covering three distinct regimes. In the subcritical range

$$
\rho<1+\frac{p}{\beta},
$$

we prove local existence and uniqueness in \(L^{p,\infty}\), together with Lipschitz dependence on the initial data and a blow-up alternative. At the critical relation

$$
\rho=1+\frac{p}{\beta},
$$

the pointwise semigroup estimate leads to a nonintegrable time singularity. By combining Lorentz duality with a Yamazaki-type integral estimate, we overcome this endpoint obstruction and obtain global mild solutions for sufficiently small initial data in \(L^{p,\infty}\).

We further treat the regime

$$
\rho>1+\frac{p}{\beta},
$$

which is supercritical relative to the prescribed space \(L^{p,\infty}\). Although the corresponding short-time Duhamel kernel is no longer integrable in this space, the mixed nature of the diffusion allows well-posedness to be recovered after imposing higher spatial integrability. Defining

$$
\frac1{\beta_\delta}
=
\frac{2\delta}{Q}+1-\frac{\alpha}{d},
$$

we show that, whenever

$$
\beta(\rho-1)<q<\beta_\delta(\rho-1),
$$

the problem is locally well posed for initial data in \(L^{q,\infty}\), while sufficiently small data generate global mild solutions. Thus the local and fractional components of the mixed operator play complementary roles, the second-order part controls the short-time behaviour, whereas the fractional component determines the admissible large-time integrability range.

These results show that Marcinkiewicz spaces provide a natural framework for combining singular initial data, degenerate Grushin diffusion, and a spatially nonlocal nonlinear source. They also reveal that the transition between the subcritical, critical, and supercritical regimes is governed not only by the nonlinear exponent but also by the integrability class of the initial datum and by the competing diffusion scales of \(G\) and \(G^\delta\). The critical and supercritical conditions considered here are understood in this Marcinkiewicz-space sense and should not be interpreted as sharp Fujita thresholds, whose determination would require corresponding nonexistence and blow-up results.

\section*{Acknowledgements} The first author acknowledges support from the ANRF-ARG MATRICS Grant(002342).\\ The second author acknowledges the financial assistance provided by the University Grants Commission (UGC), India (File No. 231610192540) during the course of the Ph.D. programme.
\section*{Conflict of Interest}
The authors declare that there are no potential competing interests.

\bibliographystyle{abbrv}
\bibliography{ref}

\end{document}